\documentclass[10pt,english,smallextended]{article}
\usepackage[T1]{fontenc}
\usepackage[latin9]{inputenc}
\usepackage[a4paper]{geometry}
\usepackage{color}
\usepackage{float}
\usepackage{algorithm2e}
\usepackage{amsmath}
\usepackage{amsthm}
\usepackage{amssymb}
\usepackage{graphicx}
\usepackage{setspace}
\makeatletter
\newcommand{\lyxaddress}[1]{
	\par {\raggedright #1
	\vspace{1.4em}
	\noindent\par}
}
\theoremstyle{plain}
\newtheorem{thm}{\protect\theoremname}
\theoremstyle{definition}
\newtheorem{defn}[thm]{\protect\definitionname}
\theoremstyle{definition}
\newtheorem{example}[thm]{\protect\examplename}
\newenvironment{lyxlist}[1]
	{\begin{list}{}
		{\settowidth{\labelwidth}{#1}
		 \setlength{\leftmargin}{\labelwidth}
		 \addtolength{\leftmargin}{\labelsep}
		 }}
	{\end{list}}
\theoremstyle{plain}
\newtheorem{lem}[thm]{\protect\lemmaname}

\date{}
\usepackage{mathtools}
\usepackage[]{xcolor}

\definecolor{blue3}{RGB}{0,0,255}
\definecolor{blue4}{RGB}{0,0,204}
\definecolor{skyblue3}{RGB}{0,102,255}
\definecolor{green4}{RGB}{0,153,0}
\definecolor{green5}{RGB}{0,102,0}
\definecolor{red5}{RGB}{153,0,0}
\definecolor{magenta4}{RGB}{204,0,204}

\makeatother

\usepackage{babel}
\providecommand{\definitionname}{Definition}
\providecommand{\examplename}{Example}
\providecommand{\lemmaname}{Lemma}
\providecommand{\theoremname}{Theorem}

\begin{document}

\title{\noindent \textbf{In how many distinct ways can flocks be formed?}\\
\textbf{A problem in sheep combinatorics}}

\author{\noindent Johanna Langner{*} and Henryk A. Witek}
\maketitle
\begin{onehalfspace}

\lyxaddress{\noindent \begin{center}
Department of Applied Chemistry and Institute of Molecular Science,
National Chiao Tung University, University Rd., 30010 Hsinchu, Taiwan\\
{*} johanna.langner@arcor.de
\par\end{center}}
\end{onehalfspace}
\begin{abstract}
In this short paper, we extend the concept of the strict order polynomial
$\Omega_{P}^{\circ}(n)$, which enumerates the number of strict order-preserving
maps $\phi:P\rightarrow\boldsymbol{n}$ for a poset $P$, to the extended
strict order polynomial $\text{E}_{P}^{\circ}(n,z)$, which enumerates
analogous maps for the elements of the power set $\mathcal{P}(P)$.
The problem at hand immediately reduces to the problem of enumeration
of linear extensions for the subposets of $P$. We show that for every
$Q\subset P$ a given linear extension $v$ of $Q$ can be associated
with a unique linear extension $w$ of $P$. The number of such linear
extensions $v$ (of length $k$) associated with a given linear extension
$w$ of $P$ can be expressed compactly as $\binom{\text{del}_{P}(w)}{k}$,
where $\text{del}_{P}(w)$ is the number of deletable elements of
$w$ defined in the text. Consequently the extended strict order polynomial
$\text{E}_{P}^{\circ}(n,z)$ can be represented as follows
\[
\text{E}_{P}^{\circ}(n,z)=\sum_{w\in\mathcal{L}(P)}\sum_{k=0}^{p}\binom{\text{del}_{P}(w)}{p-k}\binom{n+\text{des}(w)}{k}z^{k}.
\]
The derived equation can be used for example for solving the following
combinatorial problem: Consider a community of $p$ shepherds, some
of whom are connected by a master-apprentice relation (expressed as
a poset $P$). Every morning, $k$ of the shepherds go out and each
of them herds a flock of sheep. Community tradition stipulates that
each of these $k$ shepherds will herd at least one and at most $n$
sheep, and an apprentice will always herd fewer sheep than his master
(or his master's master, etc). In how many ways can the flocks be
formed? The strict order polynomial answers this question for the
case in which all $p$ shepherds go to work, and the extended strict
order polynomial considers also all the situations in which some of
the shepherds decide to take a day off.\maketitle
\end{abstract}

\section{Notation and Definitions}

\subsection{Standard terminology}

The current communication closely follows the poset terminology introduced
in Stanley's book \cite{stanley_enumerative_1986}. The reader familiar
with the terminology can jump directly to Subsection~\ref{subsec:New-terminology}.
A \emph{partially ordered set} $P$, or \emph{poset} for short, is
a set together with a binary relation $<_{P}$. In this manuscript,
we are concerned with finite posets $P$ consisting of $p$ elements
and with strict partial orders, meaning that the relation $<_{P}$
is irreflexive, transitive and asymmetric. An \emph{induced subposet}
$Q\subset P$ is a subset of $P$ together with the order $<_{Q}$
inherited from $P$ which is defined by $s<_{P}t\iff s<_{Q}t$. The
symbol $<$ shall denote the usual relation ,,larger than'' in $\mathbb{N}$.
The symbol $[\,n\,]$ stands for the set $\left\{ 1,2,\ldots,n\right\} $,
and $(n,m)$ stands for the set $\left\{ n+1,n+2,\ldots,m-1\right\} $.
The symbol $\boldsymbol{n}$ represents the chain $1<2<3<...<n$.
A map $\phi:P\rightarrow\mathbb{\mathbb{\mathbb{N}}}$ is a \emph{strict
order-preserving map} if it satisfies $s<_{P}t\Rightarrow\phi(s)<\phi(t)$.
A \emph{natural labeling} of a poset $P$ is an order-preserving bijection
$\omega:P\rightarrow[\,p\,]$. A \emph{linear extension} of $P$ is
an order-preserving bijection $\sigma:P\rightarrow\boldsymbol{p}$.
A linear extension $\sigma$ can be represented as a permutation $\omega\circ\sigma^{-1}$
expressed as a sequence $w=w_{1}w_{2}\ldots w_{p}$ with $w_{i}=\omega(\sigma{}^{-1}(i))$;
the sequence $w$ shall also be referred to as a linear extension
in the following. The set of all such sequences $w$ is denoted by
$\mathcal{L}\left(P\right)$ and is referred to as the \emph{Jordan-Hölder
set} of $P$. If two subsequent labels $w_{i}$ and $w_{i+1}$ in
$w$ stand in the relation $w_{i}>w_{i+1}$, then the index $i$ is
called a \emph{descent} of $w$. The total number of descents of $w$
is denoted by $\text{des}(w)$. The strict order polynomial $\Omega_{P}^{\circ}(n)$
of a poset $P$ \cite{stanley_chromatic-like_1970,stanley_ordered_1972,stanley_enumerative_1986},
which enumerates the strict order-preserving maps $\phi:P\rightarrow[\,n\,]$,
is given by 
\begin{equation}
\Omega_{P}^{\circ}(n)=\sum_{w\in\mathcal{L}(P)}\binom{n+\text{des}(w)}{p}.\label{eq:StrictOrderPoly}
\end{equation}

\subsection{Non-standard terminology\label{subsec:New-terminology}}

We will often construct\textemdash by slight abuse of notation\textemdash a
subposet of $P$ by specifying a set of labels $D\subset[\,p\,]$:
The expression $P\setminus D$ stands for the induced subposet with
the elements $\left\{ p\in P\,\vert\,\omega(p)\notin D\right\} $.
Clearly the subposet constructed in this way has $p-\#D$ elements;
and the full set $\mathcal{P}(P)$ of subposets of $P$ stands in
direct correspondence to the power set of $[\,p\,]$: $\mathcal{P}(P)=\left\{ P\setminus D\,\vert\,D\in\mathcal{P}([\,p\,])\right\} $.
Similarly, if $w$ is a sequence in $\mathcal{L}\left(P\right)$ and
$D$ is a subset of $[\,p\,]$, let us denote by $w\setminus D$ the
subsequence obtained by deleting all the elements of $D$ from $w$.
For example, $13245\setminus\{1,4\}=325$. Clearly, deleting some
arbitrary set $D$ from two different sequences may produce the same
sequence: for example, $13245\setminus\{1,4\}=325=32154\setminus\{1,4\}$.
We will later (Lemma~\ref{Equiv>}) see that deleting \emph{deletable}
elements (see Def.~\ref{def:deletable}) from two distinct sequences
always results in two distinct subsequences.

Further, let us slightly change the representation of linear extensions
of subposets: Normally, one would assign to each subposet $Q=P\setminus D$
a new natural labeling $\omega^{Q}:Q\rightarrow[\,q\,]$, and then
express the linear extensions of $Q$ as sequences of the elements
of $[\,q\,]$. Instead, we avoid re-labeling each subposet, and use
instead the labeling $\omega:Q\rightarrow[\,p\,]\setminus D$ inherited
from $P$. Then, a linear extension $\sigma$ of $Q$ is represented
by a sequence $w=w_{1}\ldots w_{q}$ defined in the usual way: $w_{i}=\omega(\sigma{}^{-1}(i))$.
The set of such sequences shall still be denoted by $\mathcal{L}(Q)$.
Using this notation, it is now easy to see (but properly demonstrated
later in Lemma~\ref{Equiv<}) that if $w$ is a linear extension
of $P$, then $w\setminus D$ is a linear extension of $P\setminus D$.

\section{Main results\label{sec:Main-results}}

In this short communication, we extend the concept of the strict order
polynomial $\Omega_{P}^{\circ}(n)$ given by Eq.~(\ref{eq:StrictOrderPoly})
to the extended strict order polynomial $\text{E}_{P}^{\circ}(n,z)$
given by Eq.~(\ref{eq:GenOrderPoly}), which enumerates and classifies
the totality of strict order-preserving maps $\phi:Q\rightarrow\boldsymbol{n}$
with $Q\subset P$. We show below in Theorem~\ref{Extended-order-polynomial}
that there exists a compact combinatorial expression characterizing
$\text{\emph{E}}_{P}^{\circ}(n,z)$. In the following, we shall always
assume that $P$ is a poset with with $p$ elements, a strict order
$<_{P}$ , and a natural labeling $\omega$. Subposets of $P$ are
always assumed to be induced.
\begin{defn}
The \emph{extended strict order polynomial $\text{E}_{P}^{\circ}(n,z)$}
of a poset $P$ is defined as
\begin{equation}
\text{E}_{P}^{\circ}(n,z)=\sum_{Q\subset P}\Omega_{Q}^{\circ}(n)z^{\#Q},\label{eq:GenOrderPoly}
\end{equation}
where the sum runs over all the induced subposets $Q$ of $P$.
\end{defn}

\begin{example}
Consider a family of three shepherds: Fiadh, Fiadh's father Aidan,
and Aidan's father Lorcan. Every day, some of the shepherds go out
and each herd a flock of at least one and at most $n$ sheep. Aidan
always herds more sheep than Fiadh, and Lorcan always herds more sheep
than both Fiadh and Aidan. How many possible ways are there of assigning
flock sizes to the shepherds?
\begin{figure}[H]
\noindent \begin{centering}
\includegraphics[scale=0.65]{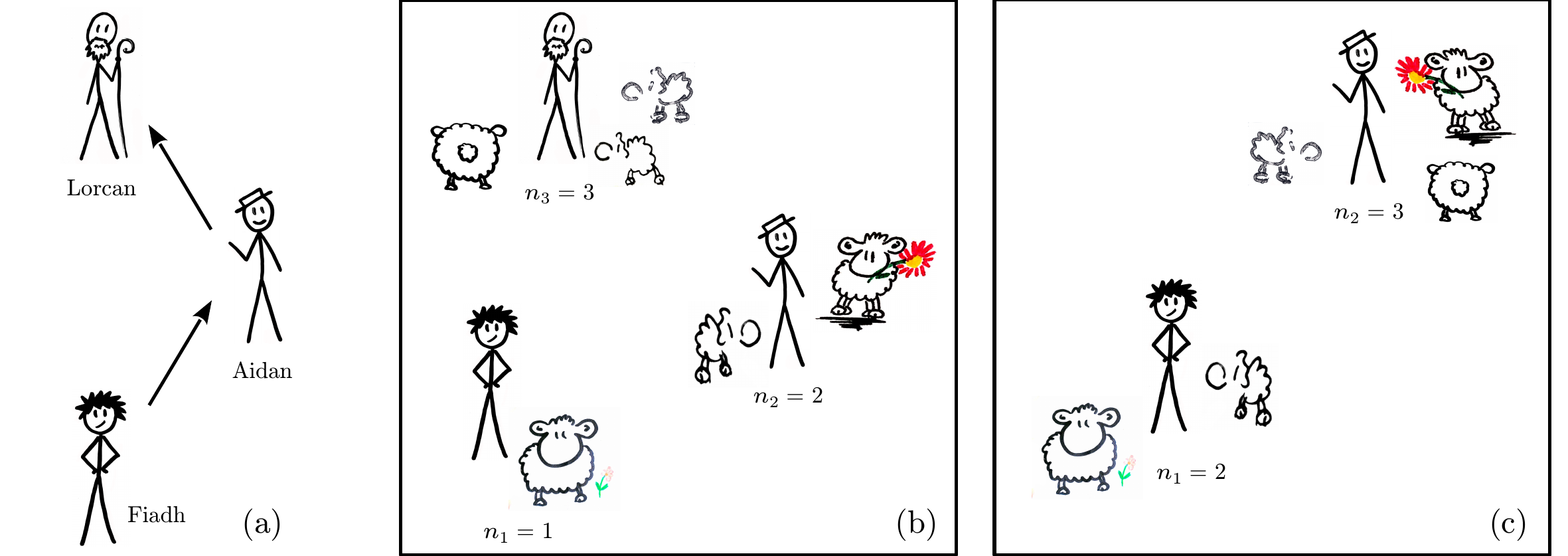}
\par\end{centering}
\caption{\label{fig:sheep}The poset formed by the three shepherds, shown in
(a), is isomorphic to the chain $\boldsymbol{3}$. Situations such
as the one depicted in (b), where all three shepherds herd a flock
of sheep, are counted by the strict order polynomial $\Omega_{P}^{\circ}(n)$.
The extended strict order polynomial $\text{E}_{P}^{\circ}(n,z)$
also counts situations such as the one shown in (c), where only a
subset of the shepherds are present.}
\end{figure}

The three shepherds together with the seniority relation form a poset
$P$ isomorphic to the chain~$\boldsymbol{3}$: $\text{Fiadh}<_{P}\text{Aidan}<_{P}\text{Lorcan}$,
see Fig.~\ref{fig:sheep}~(a). Let us denote the number of sheep
in Fiadh's flock by $n_{1}$, the size of Aidan's flock by $n_{2}$
and the size of Lorcan's flock by $n_{3}$; then the above conditions
tell us that $1\leq n_{1}<n_{2}<n_{3}\leq n$. On a day when all three
shepherds go to work, such as depicted in Fig.~\ref{fig:sheep}~(b),
the numbers $n_{1}$, $n_{2}$ and $n_{3}$ can be chosen in $\Omega_{P}^{\circ}(n)=\binom{n}{3}$
ways. When only Fiadh and Aidan go to work, see Fig.~\ref{fig:sheep}~(c),
i.e. for the subposet $Q=\{\text{Fiadh},\text{Aidan}\}$, we have
to choose the two numbers $n_{1}$ and $n_{2}$ such that $1\leq n_{1}<n_{2}\leq n$;
there are $\Omega_{Q}^{\circ}(n)=\binom{n}{2}$ ways to do so. The
same is true whenever two of the three shepherds are present. Clearly,
whenever only one shepherd works, there are $\binom{n}{1}$ ways to
choose his flock size, and when all shepherds take the day off, there
is only one possibility. Therefore, the extended strict order polynomial
$\text{E}_{P}^{\circ}(n,z)$ has the form
\begin{eqnarray}
\text{E}_{P}^{\circ}(n,z) & = & \Omega_{P}^{\circ}(n)z^{3}+\left(\Omega_{\{\text{Fiadh},\text{Aidan}\}}^{\circ}(n)+\Omega_{\{\text{Fiadh},\text{Lorcan}\}}^{\circ}(n)+\Omega_{\{\text{Aidan},\text{Lorcan}\}}^{\circ}(n)\right)z^{2}\label{eq:exEP}\\
 &  & +\left(\Omega_{\{\text{Fiadh}\}}^{\circ}(n)+\Omega_{\{\text{Aidan}\}}^{\circ}(n)+\Omega_{\{\text{Lorcan}\}}^{\circ}(n)\right)z^{1}+\Omega_{\{\text{Fiadh}\}}^{\circ}(n)z^{0}\nonumber \\
 & = & \binom{n}{3}z^{3}+3\binom{n}{2}z^{2}+3\binom{n}{1}z^{1}+\binom{n}{0}z^{0}\nonumber \\
 & = & \sum_{k=0}^{3}\binom{3}{k}\binom{n}{k}z^{k}.\nonumber 
\end{eqnarray}
The very compact expression for $\text{E}_{P}^{\circ}(n,z)$ given
in the last line of Eq.~(\ref{eq:exEP}) can be obtained directly
by applying the following theorem.
\end{example}

\begin{thm}
\label{Extended-order-polynomial}The extended strict order polynomial
is given by
\begin{equation}
\text{\emph{E}}_{P}^{\circ}(n,z)=\sum_{w\in\mathcal{L}(P)}\sum_{k=0}^{p}\binom{\text{del}_{P}(w)}{p-k}\binom{n+\text{des}(w)}{k}z^{k},\label{eq:Znz}
\end{equation}
where $\text{del}_{P}(w)$ denotes the number of deletable labels
in $w$.
\end{thm}

Intuitively speaking, deletable labels can be understood to be the
entries of $w$ which are not essential for distinguishing $w$ from
other elements of $\mathcal{L}(P)$. Theorem~\ref{Extended-order-polynomial}
is based on the fact that every element of $\bigcup_{Q\subset P}\mathcal{L}(Q)$
can be uniquely associated with some element of $\mathcal{L}(P)$.
Formally speaking, it is possible to define an equivalence relation
$\sim$ on $\bigcup_{Q\subset P}\mathcal{L}(Q)$ such that $\bigcup_{w\in\mathcal{L}(P)}\left[w\right]_{\sim}=\bigcup_{Q\subset P}\mathcal{L}(Q)$.
This concept is illustrated below in Examples~\ref{ex:2x2} and~\ref{ex:123}.
The proof of Theorem~\ref{Extended-order-polynomial} will be given
at the end of this paper after formalizing the concept of deletable
labels and proving some technical lemmata.
\begin{example}
\label{ex:2x2}Let us consider the lattice $P=\boldsymbol{2}\times\boldsymbol{2}$,
for which $\mathcal{L}(P)=\left\{ 1234,1324\right\} $ and $\bigcup_{Q\subset P}\mathcal{L}(Q)=\left\{ \varnothing,1,2,3,4,12,13,14,23,24,34,32,123,124,134,234,324,132,1234,1324\right\} $.
Our results allow us to partition the set $\bigcup_{Q\subset P}\mathcal{L}(Q)$
of linear extensions into two equivalence classes $\left[1234\right]_{\sim}$
and $\left[1324\right]_{\sim}$:
\begin{figure}[H]
\centering{}\includegraphics[scale=0.5]{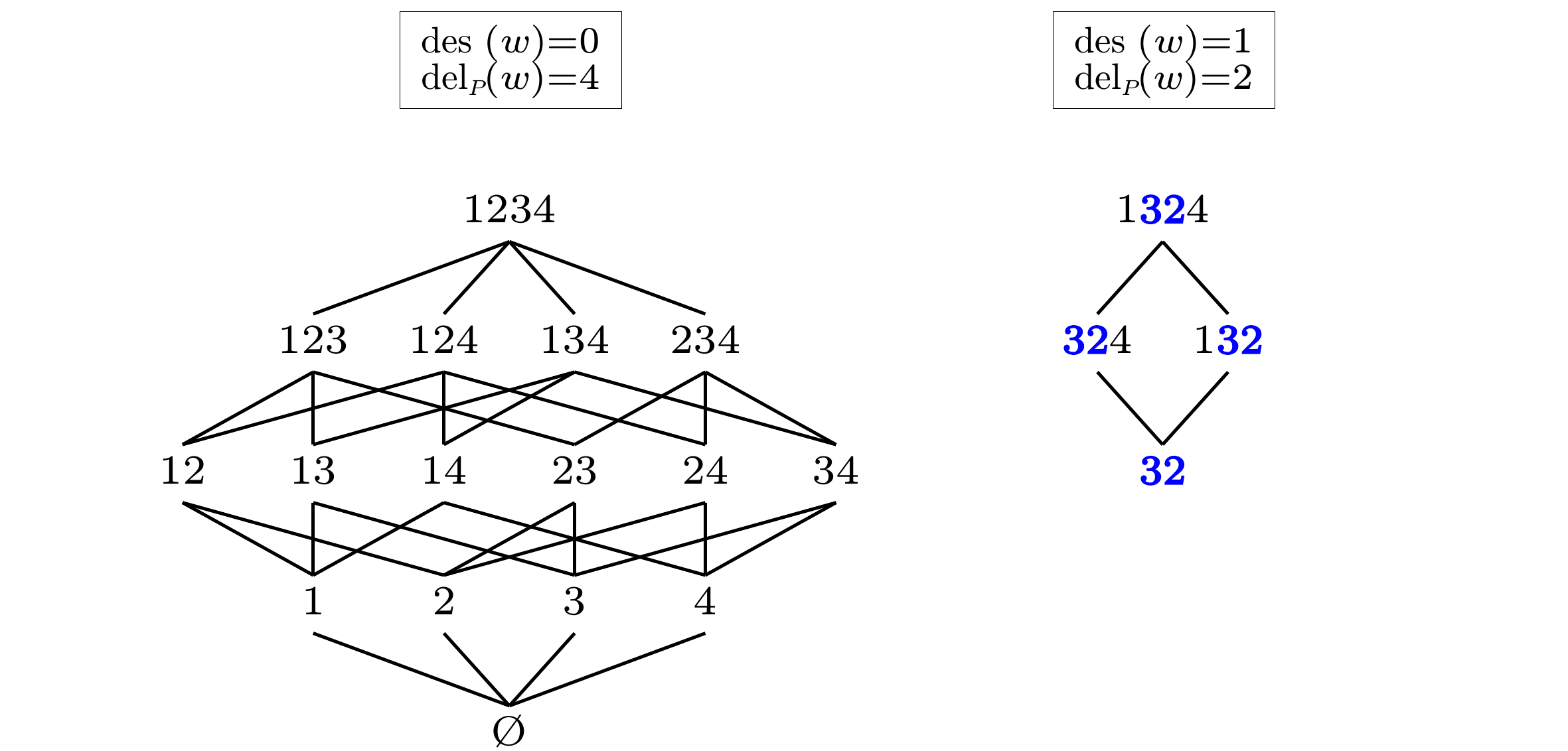}
\end{figure}
The first family, originating from the linear extension $1234$, is
characterized by zero descents ($\text{des}(w)=0$) and zero non-deletable
elements ($\text{del}_{P}(w)=4-0$), and thus contains $\binom{4-0}{k-0}$
sequences of each length $k$. The second family, originating from
the linear extension $1324$, is characterized by one descent ($\text{des}(w)=1$)
and two non-deletable elements \textcolor{blue3}{\textbf{3}} and
\textcolor{blue3}{\textbf{2}} ($\text{del}_{P}(w)=4-2$), and thus
contains $\binom{4-2}{k-2}$ sequences of each length $k$. Consequently,
the extended strict order polynomial is given by
\[
\text{E}_{P}^{\circ}(n,z)=\sum_{k=0}^{4}\left(\binom{4-0}{k-0}\binom{n}{k}+\binom{4-2}{k-2}\binom{n+1}{k}\right)z^{k}.
\]
\end{example}

\begin{example}
\label{ex:123}Let us consider the poset $P=\left\{ a,b,c\right\} $
of three non-comparable elements. We have $\mathcal{L}(P)=\left\{ 123,132,213,231,312,321\right\} $
and $\bigcup_{Q\subset P}\mathcal{L}(Q)=\left\{ \varnothing,1,2,3,12,21,13,31,23,32,123,132,213,231,312,321\right\} $.
Our results allow us to classify the linear extensions in $\bigcup_{Q\subset P}\mathcal{L}(Q)$
into six families
\begin{figure}[H]
\centering{}\includegraphics[scale=0.5]{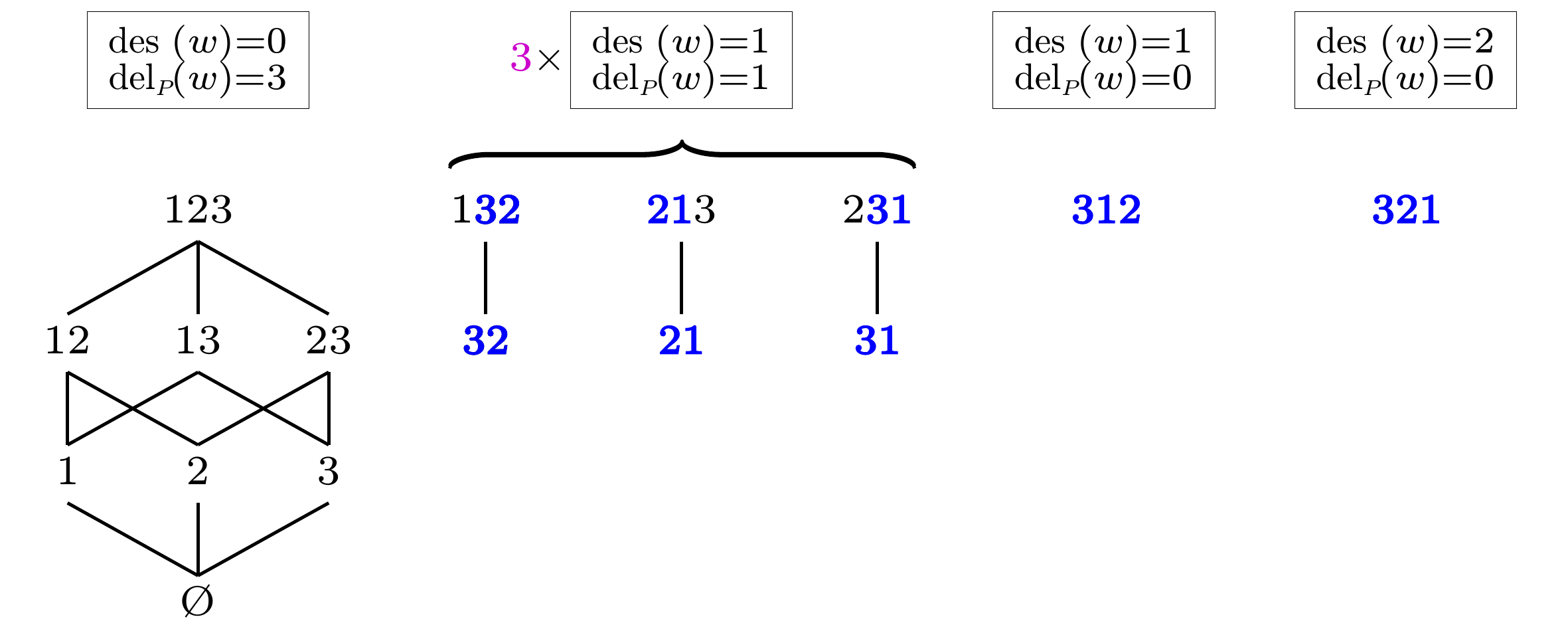}
\end{figure}
each of which is characterized by a pair of numbers $\left(\text{des}(w),\text{del}_{P}(w)\right)$
specified above. The extended strict order polynomial is given by
\[
\text{E}_{P}^{\circ}(n,z)=\sum_{k=0}^{3}\left(\binom{3}{k}\binom{n}{k}+\textcolor{magenta4}{3}\binom{3-2}{k-2}\binom{n+1}{k}+\binom{3-3}{k-3}\binom{n+1}{k}+\binom{3-3}{k-3}\binom{n+2}{k}\right)z^{k}.
\]
\end{example}

\section{Classification of linear extensions}
\begin{defn}
\label{def:deletable}Consider a sequence $w=w_{1}w_{2}\ldots w_{p}$
which represents a linear extension $\sigma$ of $P$. A label $w_{i}$
is \emph{deletable} if
\begin{lyxlist}{00.00.00}
\item [{$1)$}] neither $i-1$ nor $i$ is a descent, and
\item [{$2)$}] at least one the following is true:
\begin{lyxlist}{00.00.00}
\item [{$a)$}] $w_{j}<w_{i}$ for any $j\in(0,i)$, or
\item [{$b)$}] there exists a label $w_{k}$ with $\omega^{-1}(w_{k})<_{P}\omega^{-1}(w_{i})$
such that $w_{j}<w_{i}$ for all $j\in(k,i)$.
\end{lyxlist}
\end{lyxlist}
The set of deletable elements of $w$ is denoted by $\text{Del}_{P}(w)$,
and its cardinality by $\text{del}_{P}(w)=\#\text{Del}_{P}(w)$. Labels
that are not deletable from $w$ are called \emph{fixed} in $w$.
\end{defn}

Loosely speaking, a label is fixed if it contributes to a descent,
or if it appears after a descent even though it would also be allowed
to appear in front of it. An inclined reader might have already noticed
that every deletable element appears in a uniquely defined position
of $w$, which may be described as ,,as early as possible without
interfering with the descent pattern''. In other words, removing
a deletable element $w_{i}$ from $w$ and reinserting it at any earlier
position cannot result in a linear extension with the same labels
involved in the descent pairs. This concept constitutes the main idea
behind associating a linear extension $v\in\bigcup_{Q\subset P}\mathcal{L}(Q)$
with a unique linear extension $w\in\mathcal{L}(P)$ developed in
detail later in this communication.

Let us now demonstrate how Definition~\ref{def:deletable} can be
used in a direct manner to identify deletable and fixed labels in
linear extensions. A useful, easy-to-use-in-practice graphical reinterpretation
of Definition~\ref{def:deletable} is introduced in Example~\ref{exa:le_graph}.
\begin{example}
Consider again the poset $P=\boldsymbol{2}\times\boldsymbol{2}$ shown
in Fig.~\ref{fig:3x3 Hasse} and its two linear extensions $w=1234$
and $w'=1324$. Let us first determine which of the labels are deletable
from $1234$. We have no descents in $1234$, so the condition \emph{$1)$}
of Definition~\ref{def:deletable} is satisfied for all of the labels.
We only need to verify condition \emph{$2)$} of Definition~\ref{def:deletable}.
The label $w_{1}=1$ is deletable because the interval $\left(0,1\right)$
is empty, and therefore condition $2a)$ is vacuously satisfied. Since
$w_{1}<w_{2}<w_{3}<w_{4}$, we find also for the remaining labels
$w_{i}=2,3,4$ that condition $2a)$ is satisfied. This shows that
all four labels in the linear extension $1234$ are deletable: $\text{Del}_{P}(1234)=\left\{ 1,2,3,4\right\} $
and $\text{del}_{P}(1234)=4$. This is not the case for the linear
extension $1324$. The labels $3$ and $2$ are fixed by condition
$1)$ because the position $i=2$ is a descent. The labels $1$ and
$4$ can be shown to be deletable in exactly the same way as for $1234$.
Consequently, in the linear extension $1324$ the set of deletable
elements is $\text{Del}_{P}(1324)=\left\{ 1,4\right\} $ and $\text{del}_{P}(1324)=2$.
\end{example}

\begin{figure}[H]
\noindent \centering{}\includegraphics[scale=0.5]{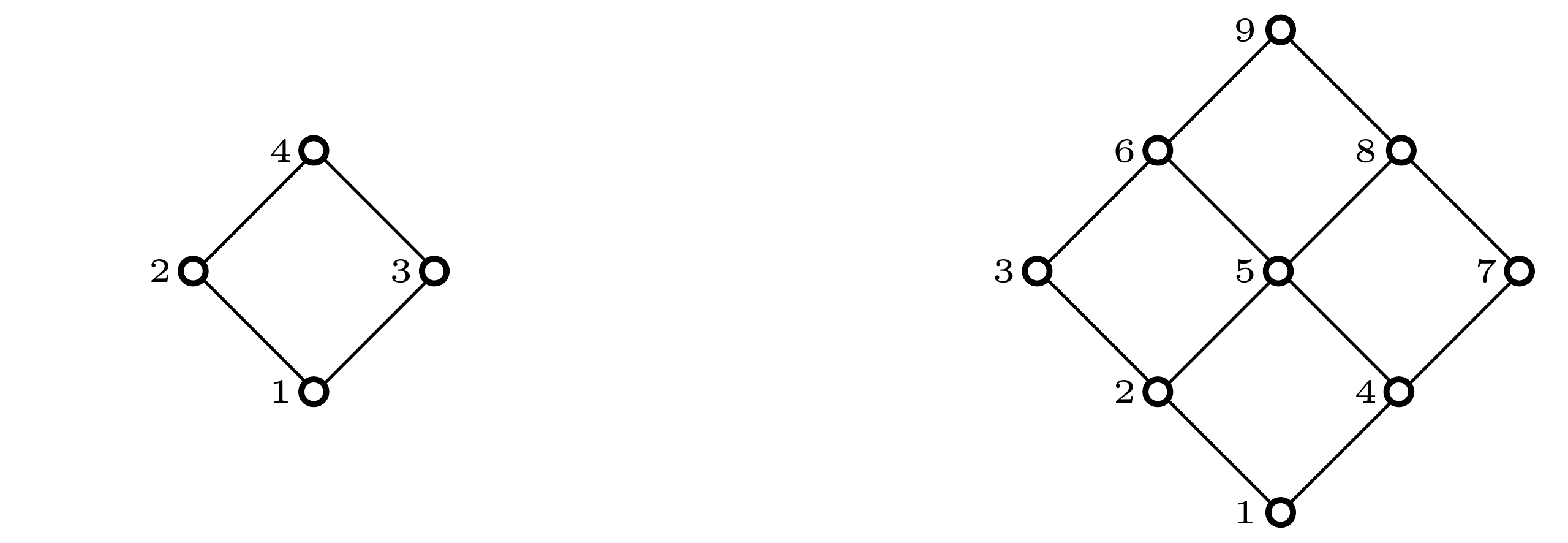}\caption{\label{fig:3x3 Hasse}Hasse diagram of two posets $P=\boldsymbol{2}\times\boldsymbol{2}$
and $P=\boldsymbol{3}\times\boldsymbol{3}$ together with (natural)
labelings $\omega$.}
\end{figure}

\begin{example}
Consider the linear extension $w=124753689$ of the poset $P=\boldsymbol{3}\times\boldsymbol{3}$
shown in Fig.~\ref{fig:3x3 Hasse}. By condition \emph{$1)$} of
Definition~\ref{def:deletable}, the labels $3,5$ and $7$ are not
deletable; the remaining labels might be deletable if they satisfy
condition $2a)$ or $2b)$ of Definition~\ref{def:deletable}. Since
$w_{1}<w_{2}<w_{3}$, the labels $w_{1}=1,w_{2}=2$ and $w_{3}=4$
are deletable by condition $2a)$; likewise, due to $w_{j}<w_{8}<w_{9}$
for all $j=1,\ldots,7$, the labels $8$ and $9$ are deletable by
condition $2a)$. Thus the only remaining label for which the deletability
(or non-deletability) is not immediately obvious\textemdash and hence
the machinery of Definition~\ref{def:deletable} must be fully put
to work\textemdash is $w_{7}=6$. There exists the label $w_{5}=5$
with $\omega^{-1}(5)<_{P}\omega^{-1}(6)$ , and for the only label
$w_{j}$ with $j\in(5,7)=\{6\}$ we find $w_{6}=3<w_{7}=6$, so by
condition $2b)$, the label $w_{7}=6$ is deletable. Therefore, $\text{Del}_{P}(w)=\left\{ 1,2,4,6,8,9\right\} $
and $\text{del}_{P}(w)=6$.
\end{example}

\begin{example}
\label{exa:le_graph}Linear extensions and their fixed and deletable
elements can be visualized and analyzed graphically in the following
way. For a given linear extension $\sigma$, plot the Hasse diagram
of $P$ in such a way that each element $t_{i}$ of $P$ is represented
as a point in a Cartesian plane with the coordinates $\left(\sigma(t_{i}),\omega(t_{i})\right)=\left(i,w_{i}\right)$,
see Fig.~\ref{fig:Graphical-representation}$(a)$. The total order
implied by the linear extension $\sigma$ is easily determined by
connecting the elements in the resulting diagram from left to right,
as shown using red arrows in Fig.~\ref{fig:Graphical-representation}$(b)$.
The fixed and deletable labels can now be identified in a graphical
way, as illustrated in Fig.~\ref{fig:Graphical-representation}$(c)$:
Whenever $i$ is a descent, i.e. whenever an element $t_{i}$ (in
the position $\left(i,w_{i}\right)$) is displayed above $t_{i+1}$
(in the position $\left(i+1,w_{i+1}\right)$), both $w_{i}$ and $w_{i+1}$
are fixed (compare condition \emph{$1)$} of Definition~\ref{def:deletable});
this is marked by coloring the corresponding elements. The line connecting
the points $\left(i,w_{i}\right)$ and $\left(i+1,w_{i+1}\right)$
then casts a ,,shadow'' to the right; and all elements in the shadow
have fixed labels as long as they are not covered by any elements
in the same shadow including $t_{i}$ and $t_{i+1}$ (if they are
covered, their labels are deletable by condition $2b)$).

\begin{figure}[H]
\noindent \begin{centering}
\begin{minipage}[t]{0.33\columnwidth}%
\noindent \begin{center}
\includegraphics[scale=0.4]{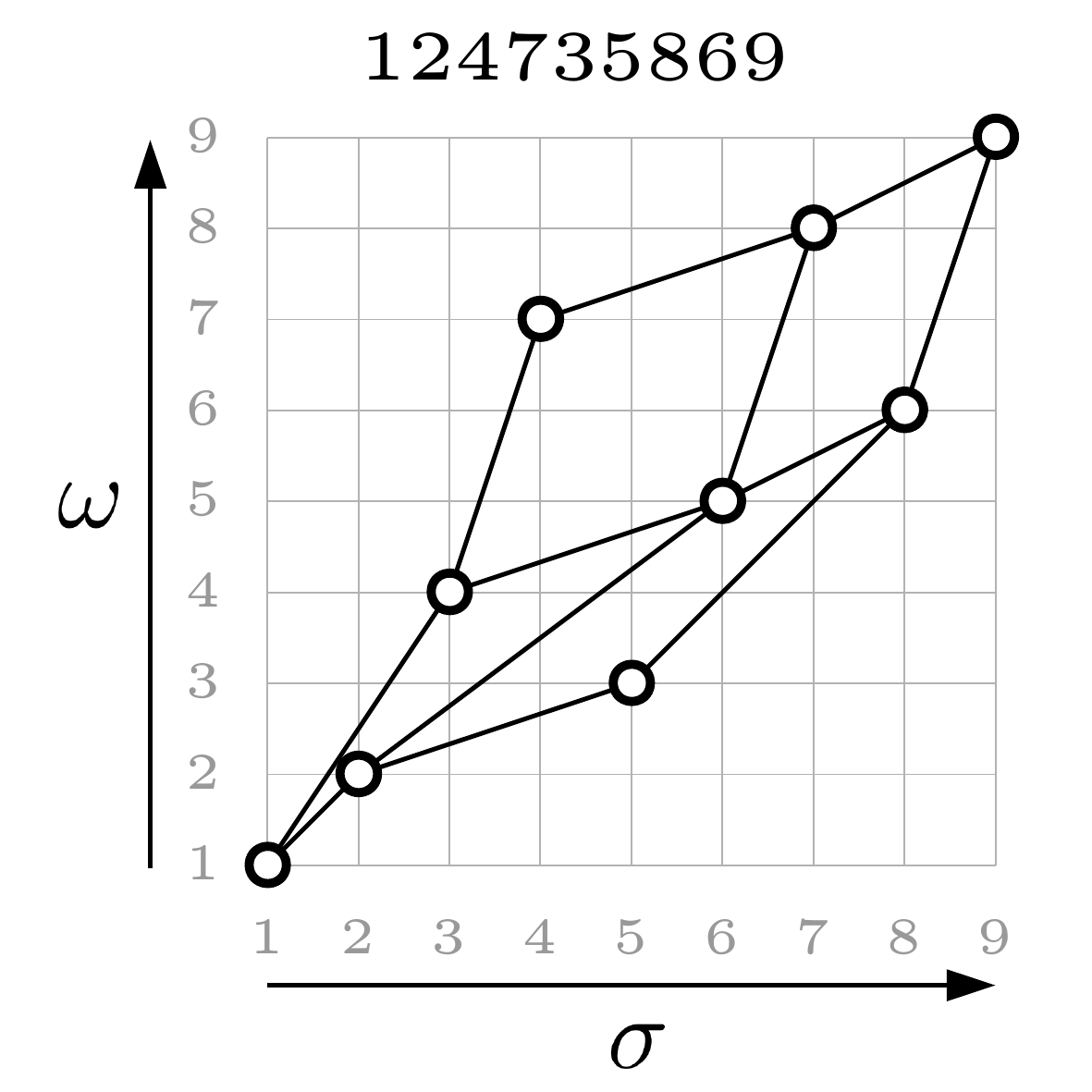}\\
$\,\,\,\,\,\,\,\,\,\,\,\,\,(a)$
\par\end{center}%
\end{minipage}%
\begin{minipage}[t]{0.33\columnwidth}%
\noindent \begin{center}
\includegraphics[scale=0.4]{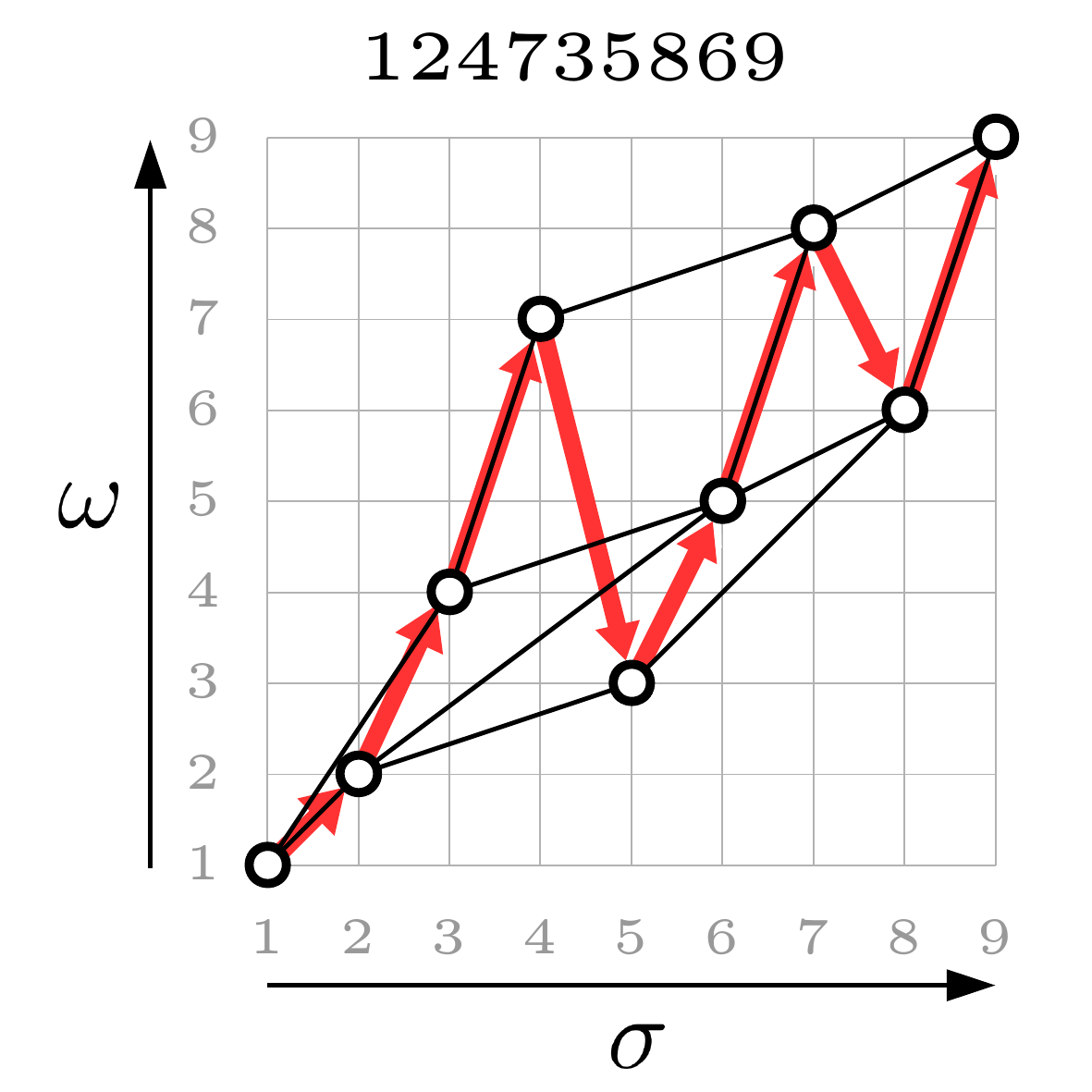}\\
$\,\,\,\,\,\,\,\,\,\,\,\,\,(b)$
\par\end{center}%
\end{minipage}%
\begin{minipage}[t]{0.33\columnwidth}%
\noindent \begin{center}
\includegraphics[scale=0.4]{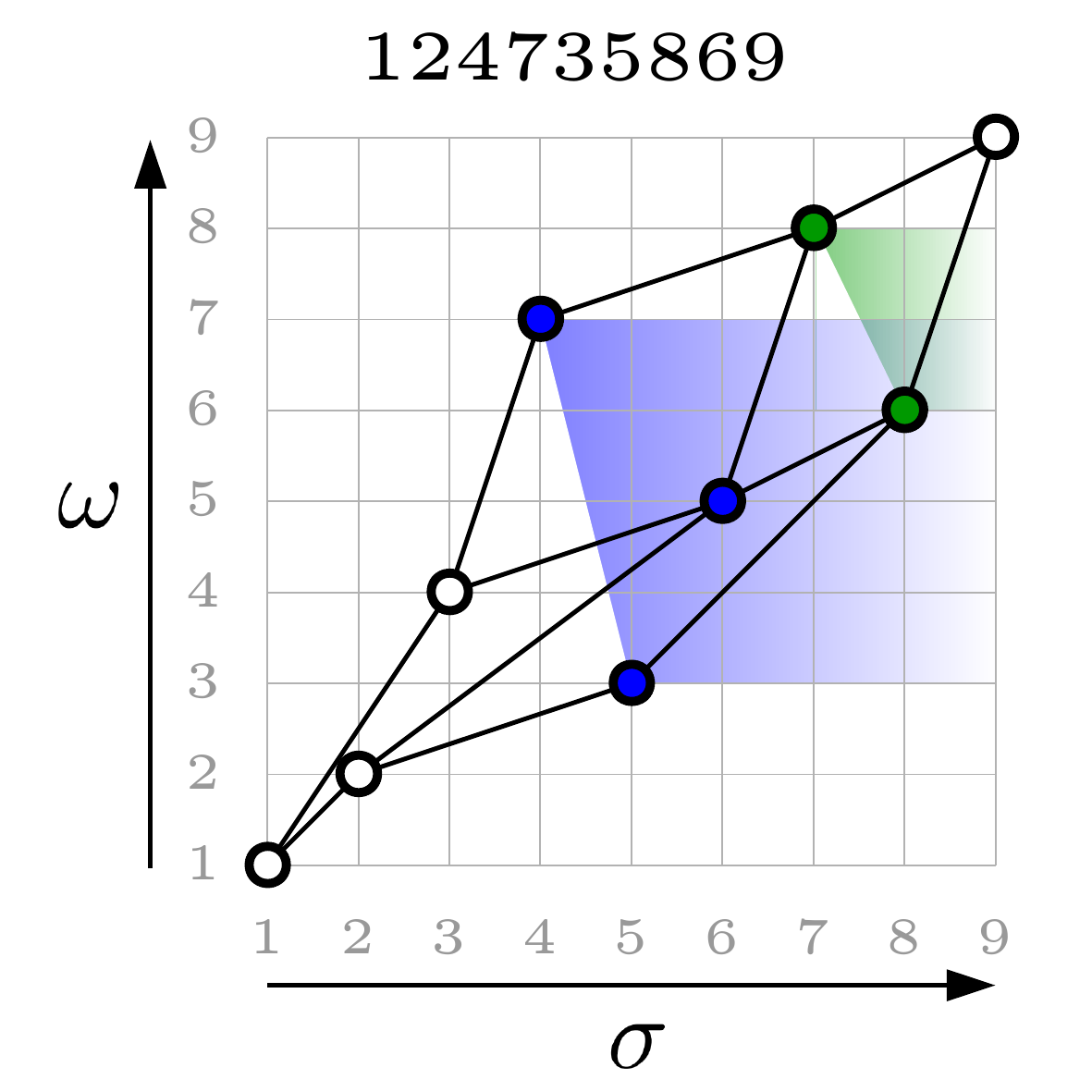}\\
$\,\,\,\,\,\,\,\,\,\,\,\,\,(c)$
\par\end{center}%
\end{minipage}
\par\end{centering}
\caption{\label{fig:Graphical-representation}Graphical representation of the
linear extension $\sigma$ represented by $w=124735869$. $(a)$ Hasse
diagram plotted in a Cartesian plane with each element $t_{i}$ of
$P$ displayed at the coordinates $\left(\sigma(t_{i}),\omega(t_{i})\right)=\left(i,w_{i}\right)$.
$(b)$ Red arrows depict the total order implied by the linear extension
$\sigma$. $(c)$ Graphical identification of elements with fixed
labels (displayed as colored circles): Every descent $i$ fixes the
labels $w_{i}$ and $w_{i+1}$. Further, the line between the two
elements $t_{i}$ and $t_{i+1}$ involved in a descent casts a shadow
to the right, and the labels of any elements caught in the shadow\textemdash and
not covered by other elements inside the same shadow, including $t_{i}$
and $t_{i+1}$\textemdash are also fixed. Elements with deletable
labels remain displayed as white circles.}
\end{figure}
\end{example}

\begin{figure}[H]
\noindent \begin{centering}
\begin{minipage}[t]{0.33\columnwidth}%
\noindent \begin{center}
\includegraphics[scale=0.4]{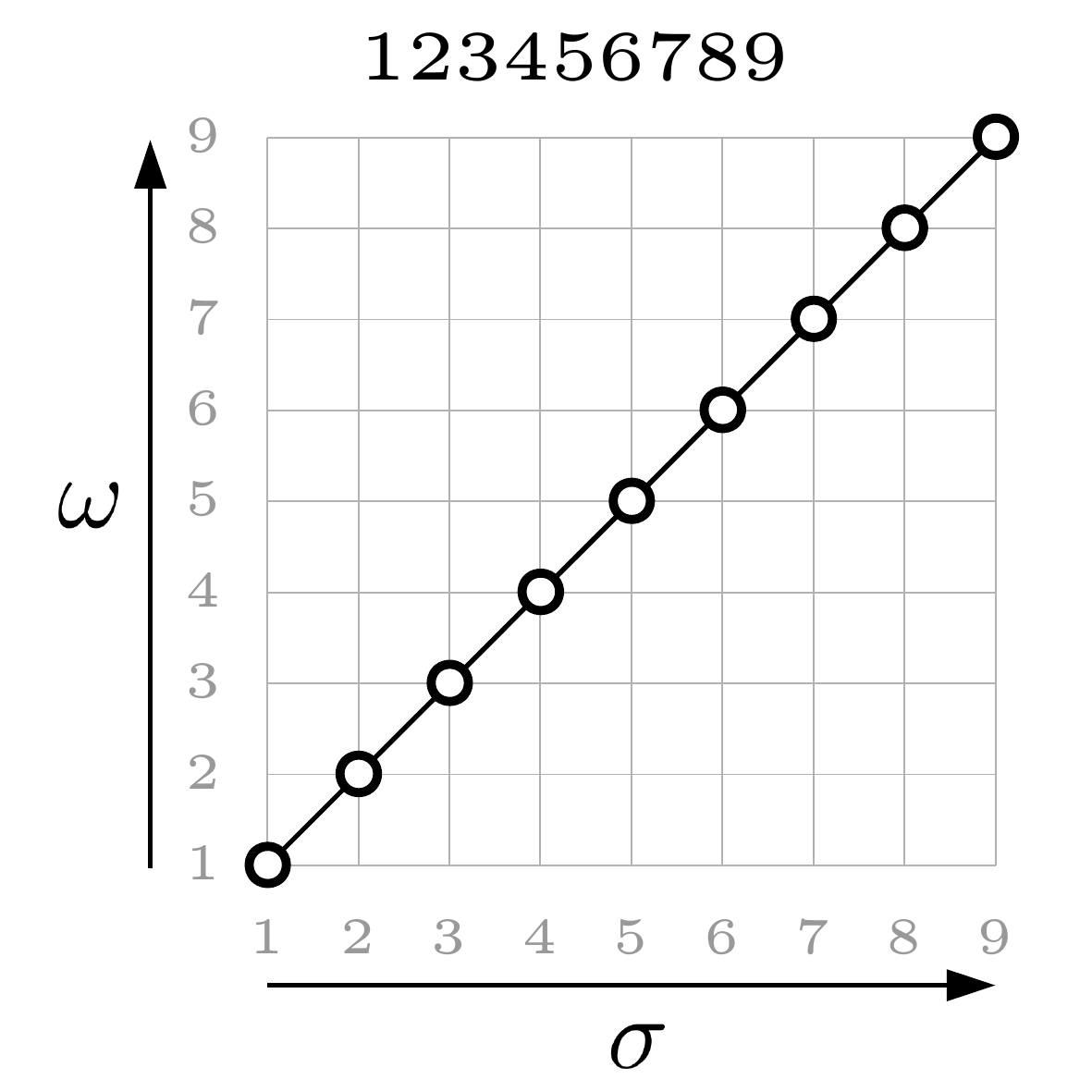}\\
$\,\,\,\,\,\,\,\,\,\,\,\,\,(a)$
\par\end{center}%
\end{minipage}%
\begin{minipage}[t]{0.33\columnwidth}%
\noindent \begin{center}
\includegraphics[scale=0.4]{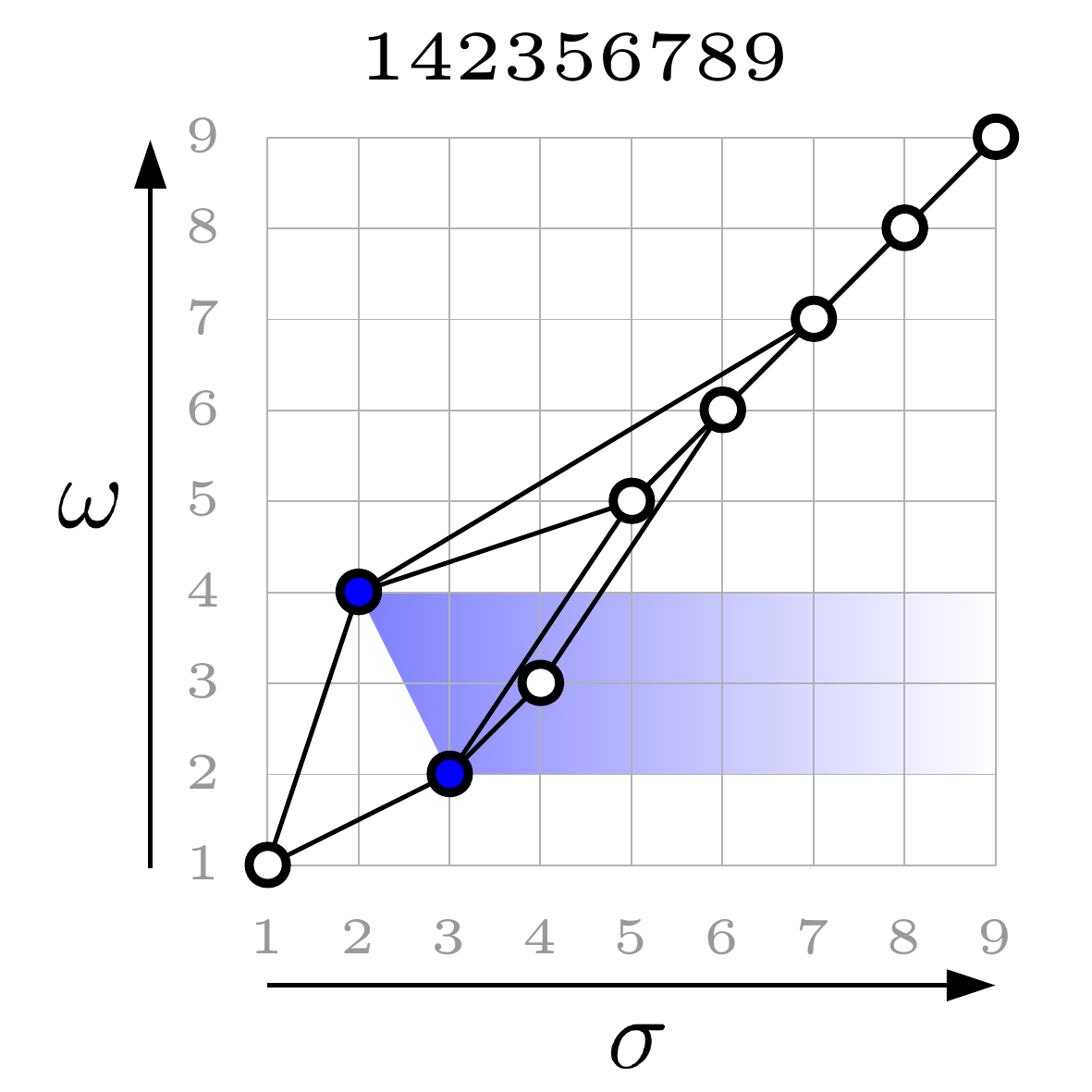}\\
$\,\,\,\,\,\,\,\,\,\,\,\,\,(b)$
\par\end{center}%
\end{minipage}%
\begin{minipage}[t]{0.33\columnwidth}%
\noindent \begin{center}
\includegraphics[scale=0.4]{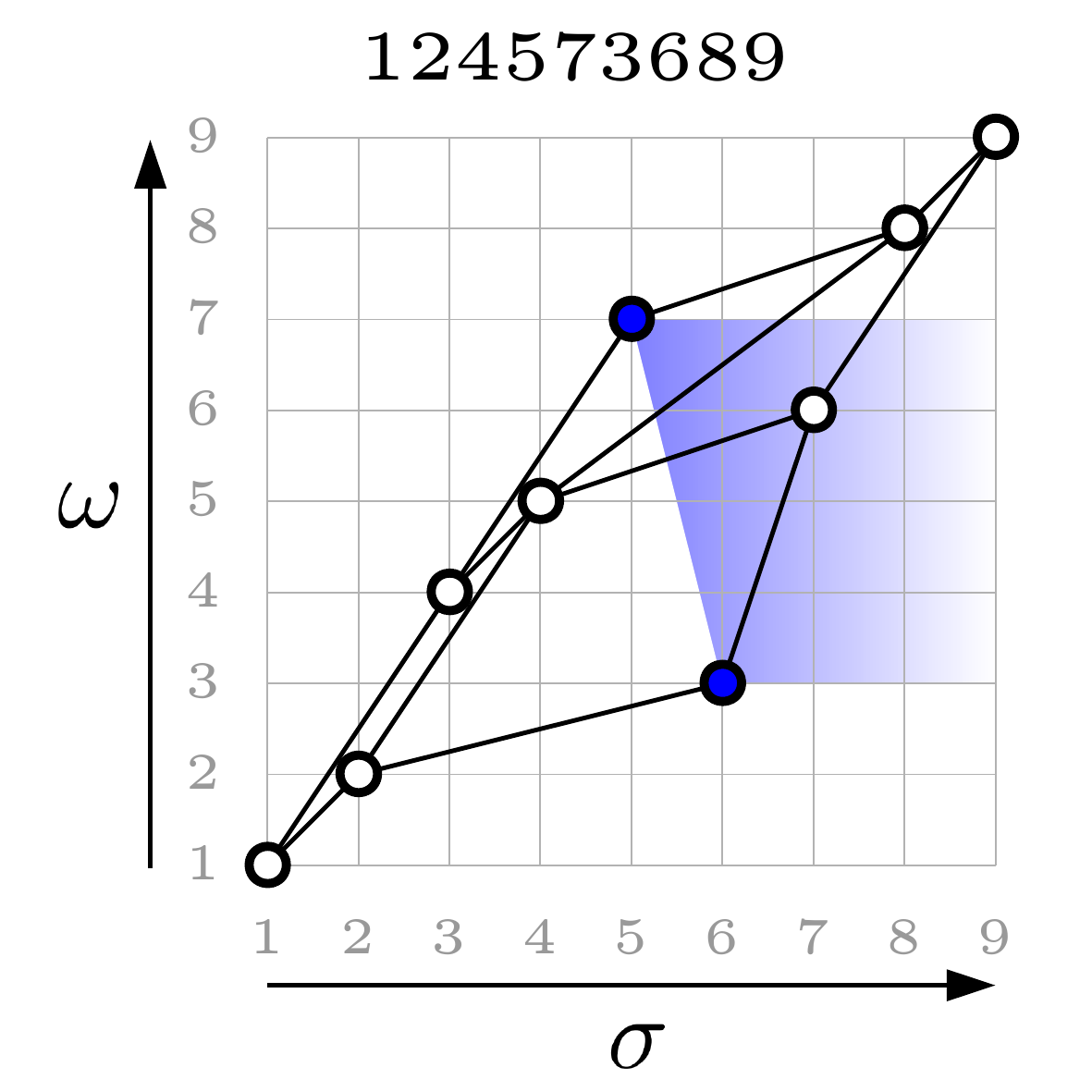}\\
$\,\,\,\,\,\,\,\,\,\,\,\,\,(c)$
\par\end{center}%
\end{minipage}\\
~
\par\end{centering}
\noindent \begin{centering}
\begin{minipage}[t]{0.33\columnwidth}%
\noindent \begin{center}
\includegraphics[scale=0.4]{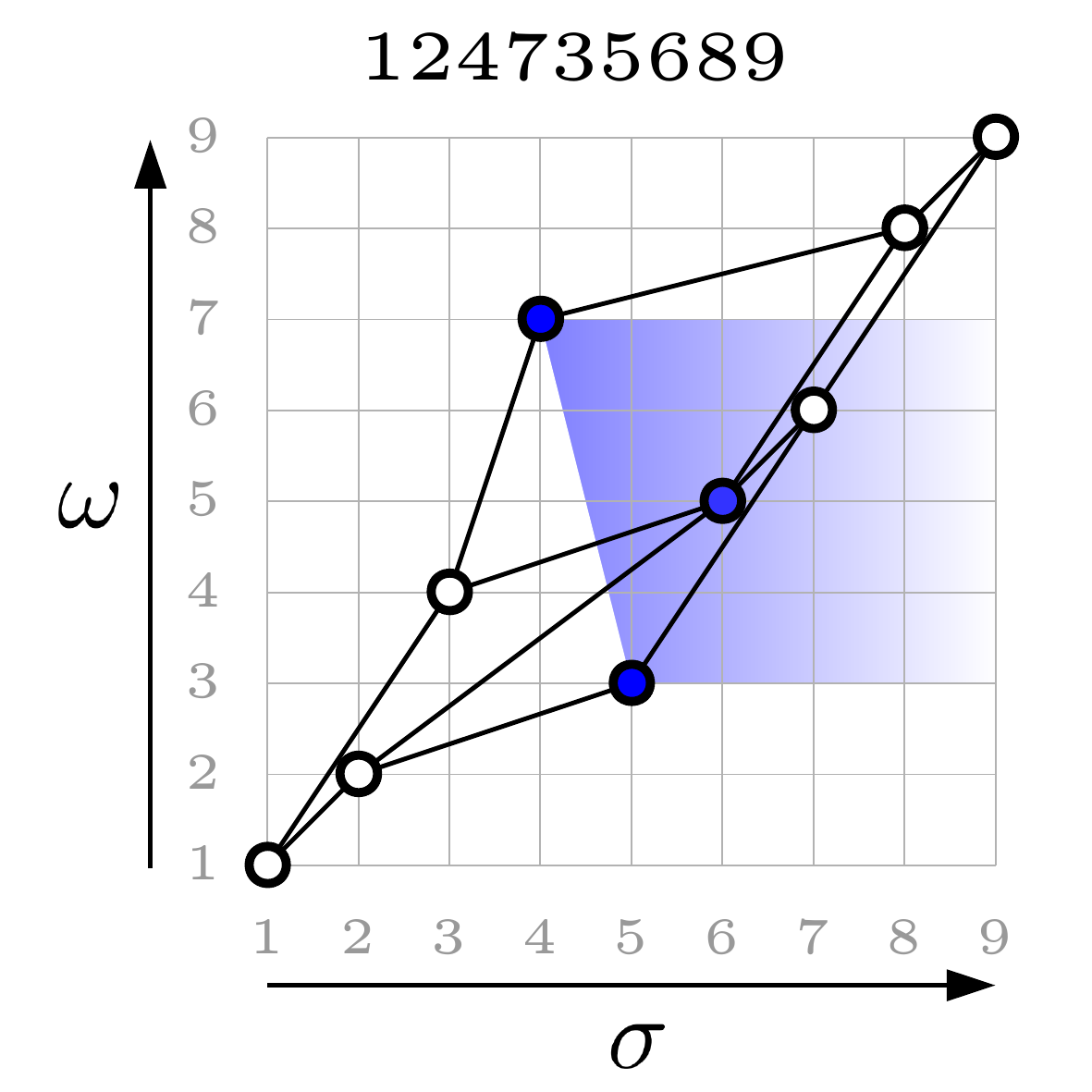}\\
$\,\,\,\,\,\,\,\,\,\,\,\,\,(d)$
\par\end{center}%
\end{minipage}%
\begin{minipage}[t]{0.33\columnwidth}%
\noindent \begin{center}
\includegraphics[scale=0.4]{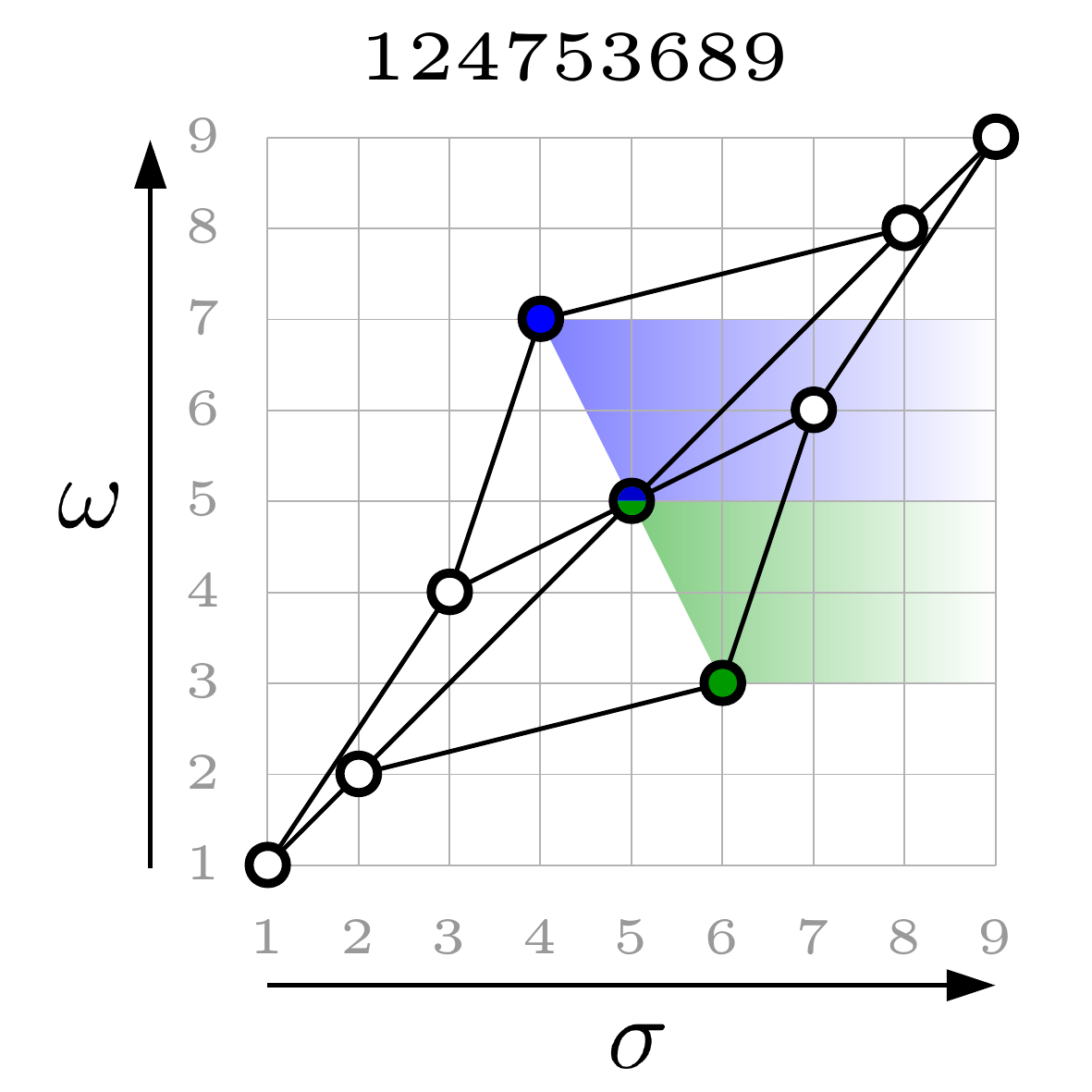}\\
$\,\,\,\,\,\,\,\,\,\,\,\,\,(e)$
\par\end{center}%
\end{minipage}%
\begin{minipage}[t]{0.33\columnwidth}%
\noindent \begin{center}
\includegraphics[scale=0.4]{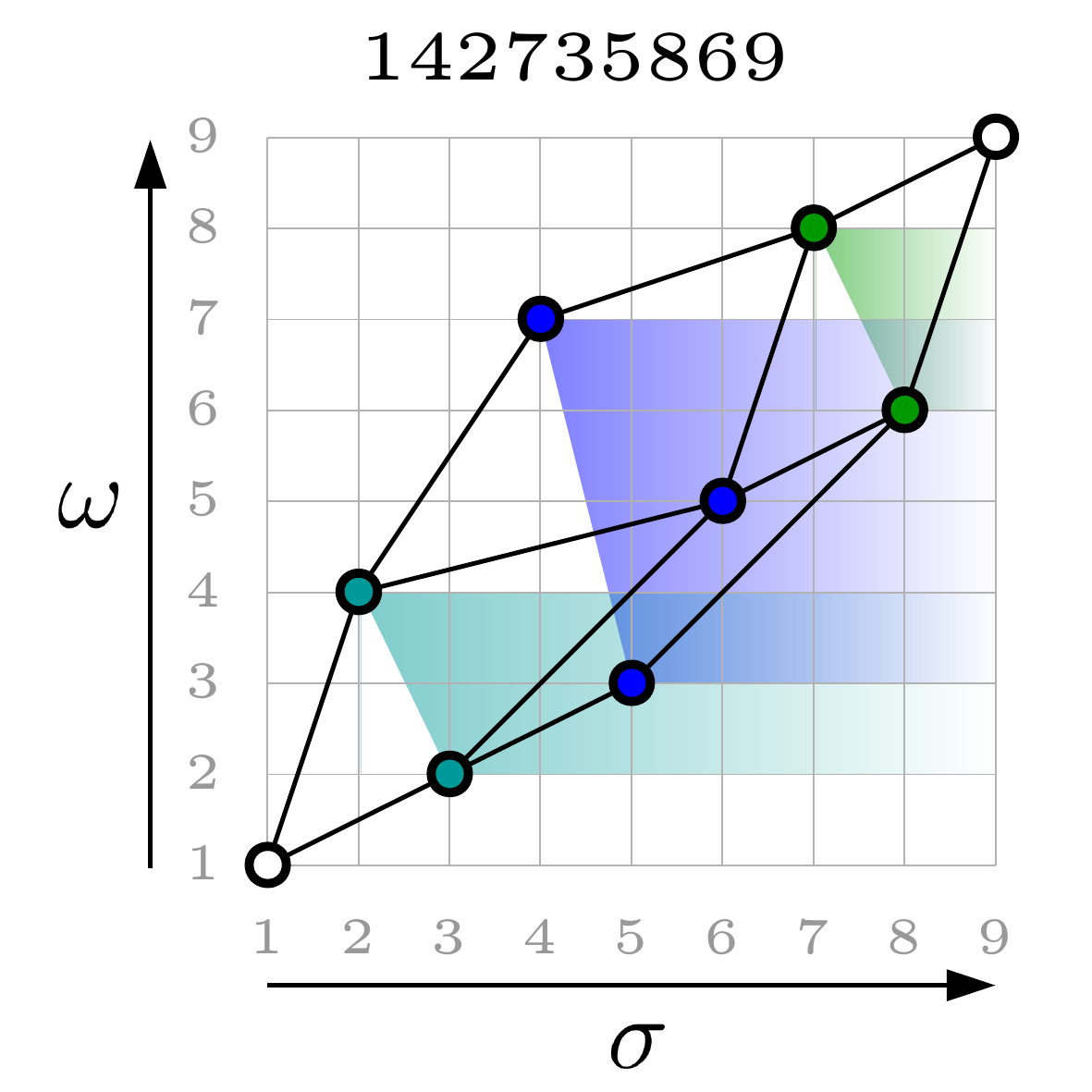}\\
$\,\,\,\,\,\,\,\,\,\,\,\,\,(f)$
\par\end{center}%
\end{minipage}
\par\end{centering}
\caption{\label{fig:six-linear-extensions}Graphical representation of six
linear extensions of the poset $P=\boldsymbol{3}\times\boldsymbol{3}$
shown in Fig.~\ref{fig:3x3 Hasse}. Elements with fixed labels are
shown as colored circles.}
\end{figure}
\begin{example}
Fig.~\ref{fig:six-linear-extensions} shows six of the linear extensions
of the poset $P=\boldsymbol{3}\times\boldsymbol{3}$ shown in Fig.~\ref{fig:3x3 Hasse}.
The introduced graphical representation of each of the linear extensions
allows us to identify easily the sets of deletable labels: $(a)$
$\text{Del}_{P}(w)=[\,9\,]$, $(b)$ $\text{Del}_{P}(w)=[\,9\,]\setminus\left\{ 2,4\right\} $,
$(c)$ $\text{Del}_{P}(w)=[\,9\,]\setminus\left\{ 3,7\right\} $,
$(d)$ $\text{Del}_{P}(w)=[\,9\,]\setminus\left\{ 3,5,7\right\} $,
$(e)$ $\text{Del}_{P}(w)=[\,9\,]\setminus\left\{ 3,5,7\right\} $,
$(f)$ $\text{Del}_{P}(w)=\left\{ 1,9\right\} $. Note that in cases
$(b)-(e)$, there are elements in the shadow which are not colored
because they are covered by another element in the same shadow \textendash{}
that is, elements that are deletable by condition $2b)$ of Definition~\ref{def:deletable}.
\end{example}

We are now ready to investigate formally the relation between the
linear extensions of $P$ and the linear extensions of its subposets
$P\setminus D$.

\subsection{Correspondence between linear extensions of a poset and of its subsets}

Deleting deletable elements does not affect the number of descents:
\begin{lem}
\label{same-nr-descents}Consider a linear extension $w\in\mathcal{L}(P)$
and a subsequence $v=w\setminus D$, where $D\subset\text{\emph{del}}_{P}(w)$.
Then, $\text{\emph{des}}(w)=\text{\emph{des}}(v)$.
\end{lem}

\begin{proof}
Let us augment the linear extension $w$ with two auxiliary fixed
labels $w_{0}=0$ and $w_{p+1}=p+1$. Then any deletable label of
$w$ is located between two fixed labels $w_{i}$ and $w_{j}$, which
can be selected in such a way that all the labels $w_{i+1},\ldots,w_{j-1}$
in between are deletable. If there is any $k\in(i,j)$ such that $w_{k}>w_{k+1}$,
$k$ would be a descent in $w$ and $w_{k}$ and $w_{k+1}$ would
be fixed according to condition $1)$ of Def.~\ref{def:deletable},
contradicting the choice of $w_{i}$ and $w_{j}$. Therefore, we have
$w_{i}<w_{i+1}<\ldots<w_{j-1}<w_{j}$. Every deletable element belongs
to such an interval containing monotonously increasing deletable labels
flanked by two fixed labels, therefore constructing $v=w\setminus D$
by deleting any deletable elements from $w$ does not remove or introduce
any descents.
\end{proof}
The following two lemmata establish the correspondence between the
elements of $\mathcal{L}(P\setminus D)$ and the elements of $\mathcal{L}(P)$
.
\begin{lem}
\label{Equiv>}Let $D\subset[\,p\,]$, and let $Q=P\setminus D$ be
a subposet of $P$. For every linear extension $v$ of $Q$ there
exists exactly one linear extension $w$ of $P$ such that $v=w\setminus D$
and $D\subset\text{\emph{del}}_{P}(w)$.
\end{lem}

\begin{proof}
Denote in the following by $q=\#Q=p-\#D$ the length of $v$. For
the sake of brevity of the following expositions, assume during this
proof that $v_{0}=0$ and $v_{q+1}=p+1$. Let us attempt to construct
a sequence $w$ of elements of $\left[\,p\,\right]$ such that
\begin{lyxlist}{00.00.0000}
\item [{$a)$}] $v=w\setminus D$,
\item [{$b)$}] $w\in\mathcal{L}(P)$,
\item [{$c)$}] $D\subset\text{Del}_{P}(w)$.
\end{lyxlist}
A sequence $w$ can only satisfy condition $a)$ if it contains the
labels $v_{1},\ldots,v_{q}$ appearing in the same order as in $v$,
preceded, interleaved, and/or succeeded by the elements of $D.$ Let
us denote by $D^{0}$ the set of elements of $D$ appearing in $w$
before $v_{1}$, by $D^{1}$ the set of elements appearing between
$v_{1}$ and $v_{2}$, and so on. Clearly, $D$ is a disjoint union
of the subsets $D^{0},D^{1},\ldots,D^{q}$. We can thus construct
a sequence $w$ in two steps:
\begin{lyxlist}{00.00.0000}
\item [{Step~1:}] Partition $D$ into $q+1$ (possibly empty) subsets
$D^{0},D^{1},\ldots,D^{q}$ containing $m_{0},m_{1},\ldots,m_{q}$
elements, respectively.
\item [{Step~2:}] Arrange the elements of each subset $D^{i}$ into a
subsequence $d_{1}^{i}d_{2}^{i}\ldots d_{m_{i}}^{i}$ and form a sequence
$w$ by concatenating the labels in $v$ and the consecutive subsequences
$d_{1}^{0}d_{2}^{0}\ldots d_{m_{0}}^{0},\ldots,d_{1}^{q}d_{2}^{q}\ldots d_{m_{q}}^{q}$
in the following way
\[
w=d_{1}^{0}d_{2}^{0}\ldots d_{m_{0}}^{0}v_{1}d_{1}^{1}d_{2}^{1}\ldots d_{m_{1}}^{1}v_{2}\ldots v_{q}d_{1}^{q}d_{2}^{q}\ldots d_{m_{q}}^{q}.
\]
\end{lyxlist}
\noindent Obviously, many different sequences $w$ can be constructed
in this way by choosing different partitionings of $D$ and by selecting
distinct orders of the elements in each $D^{i}$; we show during the
following construction process that the conditions $b)$ and $c)$
restrict this abundance to a single, unique sequence $w$. 

\noindent Every $d\in D$ must be inserted in such a way that that
$w_{i-1}<w_{i}\equiv d<w_{i+1}$, otherwise, $d$ would\textemdash by
condition $1)$ of Definition~\ref{def:deletable}\textemdash not
be deletable from $w$, thus violating condition $c)$. Therefore,
the subsequence $d_{1}^{i}d_{2}^{i}\ldots d_{m_{i}}^{i}$ inserted
between $v_{i}$ and $v_{i+1}$ must satisfy $v_{i}<d_{1}^{i}<d_{2}^{i}<\ldots<d_{m_{i}}^{i}<v_{i+1}$.
This shows that, in Step 2, when we augment $v$ with the elements
of a subset $D^{i}$, the only choice is to arrange these elements
into a monotonously increasing sequence before doing so. Moreover,
this requirement seriously reduces the number of allowed partitions
of $D$ into subsets $D^{i}$, as each $d\in D^{i}$ needs to satisfy
the condition $v_{i}<d<v_{i+1}$.

Consider an element $d\in D$. Let us now narrow down the family of
subsets $D^{i}$ into which the element $d$ may be placed. Let $j_{d}=\text{max}\left\{ j\in[\,q\,]\,\vert\,\omega^{-1}(v_{j})<_{P}\omega^{-1}(d)\right\} $,
or $j_{d}=0$ if this set is empty. In order to not violate condition
$b)$, $d$ must be in some $D^{i}$ with $j_{d}\leq i$. Denote by
$I_{d}=\left\{ i\,\vert\,i\geq j_{d},\,v_{i}<d<v_{i+1}\right\} $
the set of possible choices for $i$ limited by the so far derived
conditions $i\geq j_{d}$ and $v_{i}<d<v_{i+1}$. The set $I_{d}$
is nonempty: It follows from the order-preserving nature of $\omega$
that $v_{j_{d}}<d<v_{q+1}$, so there must be at least one value of
$i$ with $j_{d}\leq i\leq q$ such that $v_{i}<d<v_{i+1}$. Let $i_{d}=\text{min }I_{d}$.

We will now show by \emph{reductio ad absurdum} that placing $d$
into a subset other than $D^{i_{d}}$ leads to a violation of condition
$c)$. (Recollect that every deletable element appears in $w$ ,,as
early as possible''.) Assume that $d\in D^{i}$ with $i\in I_{d}$
and $i>i_{d}$. Then, by definition of $I_{d}$, we know that $d<v_{i_{d}+1}$.
In order for $d$ to be deletable from the sequence $w$, there must
be an element $e\in[\,p\,]$ such that $\omega^{-1}(e)<_{P}\omega^{-1}(d)$
and which appears in $w$ between $v_{i_{d}+1}$ and $d$. Denote
by $E$ the set of such elements: $E=\left\{ e\in[\,p\,]\,\vert\,\omega^{-1}(e)<_{P}\omega^{-1}(d),\sigma(\omega^{-1}(v_{i_{d}+1}))<\sigma(\omega^{-1}(e))<\sigma(\omega^{-1}(d))\right\} $,
where $\sigma$ denotes the map $\sigma:P\rightarrow\boldsymbol{p}$,
$w_{i}\mapsto i$ implied by $w$. If there is an $e\in E$ with $e\notin D$,
then $e$ must appear in $v$ at some position $k$, $e\equiv v_{k}$.
Since $\omega^{-1}(e)<\omega^{-1}(d)$, according to the definitions
of $j_{d}$ and $i_{d}$ we find that $k\leq j_{d}\leq i_{d}$, in
contradiction with the requirement that $v_{i_{d}+1}$ precede $e=v_{k}$
in $v$. Therefore, $e\in D$ and thus $E\subset D$. Consider now
the element $c=\text{min }E$. It follows from $\omega^{-1}(c)<_{P}\omega^{-1}(d)$
that $c<d<v_{i_{d}+1}$. In order for $c$ to be deletable from the
finished sequence $w$, there must be an element $e'\in[\,p\,]$ such
that $\omega^{-1}(e')<_{P}\omega^{-1}(c)$ and which appears in $w$
between $v_{i_{d}+1}$ and $c$. Because of $\omega^{-1}(c)<\omega^{-1}(d)$
and since $c$ must precede $d$ in $w$, the aforementioned element
$e'$ must be in $E$, and therefore $\text{min }E=c<e'$. At the
same time, since $\omega$ is a natural labeling, $\omega^{-1}(e')<_{P}\omega^{-1}(c)$
implies $e'<c$. This contradiction shows that $c$ cannot be deletable
from $w$, $c\notin\text{Del}_{P}(w)$. However we have found before
that $c\in D$, which means that the assumption $i>i_{d}$ leads to
a violation of condition $c)$. Therefore, we must have $i\leq i_{d}$.
Since $i_{d}$ is defined as the minimum allowed value of $i,$ we
have $i=i_{d}$.

\noindent To summarize, we have shown until now that the only way
to construct a sequence $w$ in a way that does not contradict conditions
$a)-c)$ is to follow the construction introduced above, which can
be described in the following way:
\begin{lyxlist}{00.00.0000}
\item [{Step~1:}] For every $d\in D$, let
\begin{eqnarray}
\phantom{\text{and }}j_{d} & = & \text{max}\left(\left\{ j\in[\,q\,]\,\vert\,\omega^{-1}(v_{j})<_{P}\omega^{-1}(d)\right\} \cup\left\{ 0\right\} \right)\label{eq:jd}\\
\text{and }i_{d} & = & \text{min}\left(\left\{ i\,\vert\,j_{d}\leq i\leq q,\,v_{i}<d<v_{i+1}\right\} \right),\label{eq:id}
\end{eqnarray}
and assign $d$ to the set $D^{i_{d}}$.
\item [{Step~2:}] For every $0\leq i\leq q$, insert between $v_{i}$
and $v_{i+1}$ the elements of $D^{i}$ in increasing order: 
\[
w=d_{1}^{0}d_{2}^{0}\ldots d_{m_{0}}^{0}v_{1}d_{1}^{1}d_{2}^{1}\ldots d_{m_{1}}^{1}v_{2}\ldots v_{q}d_{1}^{q}d_{2}^{q}\ldots d_{m_{q}}^{q},
\]
where $d_{q}^{i}\in D^{i}$ and $v_{i}<d_{1}^{i}<d_{2}^{i}<\ldots<d_{m_{i}}^{i}<v_{i+1}$.
\end{lyxlist}
\noindent It remains to be demonstrated that the sequence $w$ uniquely
defined in this way indeed satisfies all conditions $a)-c)$. Condition
$a)$ is satisfied by construction.

Next, let us verify that $w$ satisfies condition $b)$. Consider
two arbitrary elements $s,t\in P$ such that $s<_{P}t$. Each of their
labels $\omega(s)$ and $\omega(t)$ can be in $D$ or in $[\,p\,]\setminus D$.
For each case, we have to show that $\omega(s)$ precedes $\omega(t)$
in $w$.
\begin{itemize}
\item If $\omega(s),\omega(t)\in[\,p\,]\setminus D$, then $\omega(s)\equiv v_{i}$
and $\omega(t)\equiv v_{j}$ for some $i,j\in[\,q\,]$. Since $Q$
is an induced subposet of $P$, we have $s<_{Q}t$, and therefore
we know that $i<j$, i.e., $v_{i}\equiv\omega(s)$ precedes $v_{j}\equiv\omega(t)$
in $v$. Then, by construction, $\omega(s)$ precedes $\omega(t)$
also in $w$.
\item If $\omega(s)\in[\,p\,]\setminus D$ and $\omega(t)\in D$, then $\omega(s)\equiv v_{k}$
for some $k\in[\,q\,]$. Step 1 defines two numbers $j_{\omega(t)}$
and $i_{\omega(t)}$. Since $s<_{P}t$, $k$ is in $\left\{ j\in[q]\,\vert\,\omega^{-1}(v_{j})<_{P}t\right\} $,
and thus by Eq.~(\ref{eq:jd}) $k\leq j_{\omega(t)}$. From Eq.~(\ref{eq:id})
it is clear that $j_{\omega(t)}\leq i_{\omega(t)}$. Consequently,
the label $\omega(t)$ is assigned to $D^{i_{\omega(t)}}$ with $k\leq i_{\omega(t)}$,
which means that $\omega(t)$ appears in $w$ after $\omega(s)\equiv v_{k}$.
\item If $\omega(s)\in D$ and $\omega(t)\in[\,p\,]\setminus D$, then $\omega(t)\equiv v_{k}$
for some $k\in[\,q\,]$. Any $v_{l}\in[\,p\,]\setminus D$ with $\omega^{-1}(v_{l})<_{P}s$
also satisfies $\omega^{-1}(v_{l})<_{P}s<_{P}t=\omega^{-1}(v_{k})$,
and therefore $l<k$. Therefore, application of Step~$1$ to $\omega(s)$
results in $j_{\omega(s)}<k$ and, due to the fact that $\omega(s)<\omega(t)$,
we have $i_{\omega(s)}<k$. Thus, $\omega(s)$ is assigned to a $D^{i_{\omega(s)}}$
with $i_{\omega(s)}<k$, and therefore it appears in $w$ before $\omega(t)$.
\item If $\omega(s),\omega(t)\in D$, then for any $v_{k}$ with $\omega^{-1}(v_{k})<_{P}s$,
it follows directly that $\omega^{-1}(v_{k})<_{P}t$. Therefore, in
Step~$1$, we find $j_{\omega(s)}\leq j_{\omega(t)}$, and as a result,
in addition to $j_{\omega(s)}\leq i_{\omega(s)}$ (due to Eq.~(\ref{eq:id}))
we also know that $j_{\omega(s)}\leq i_{\omega(t)}$. By construction
of $j_{\omega(s)}$ and $i_{\omega(t)}$ as well as the order-preserving
nature of $\omega$, we find $v_{j_{\omega(s)}}<\omega(s)<\omega(t)<v_{i_{\omega(t)}+1}$.
Therefore, there must be at least one value of $i$ in the interval
$j_{\omega(s)},\ldots,i_{\omega(t)}$ such that $v_{i}<\omega(s)<v_{i+1}$.
It follows that $i_{\omega(s)}\leq i_{\omega(t)}$. If $i_{\omega(s)}<i_{\omega(t)}$,
then obviously $\omega(s)$ appears in $w$ before $v_{i_{\omega(t)}}$,
which in turn appears before $\omega(t)$. Finally, even if $i_{\omega(s)}=i_{\omega(t)}$,
in Step~$2$ the elements of each $D^{i}$ are inserted into the
sequence $w$ in increasing order, so in any case $\omega(s)$ will
be inserted before $\omega(t)$.
\end{itemize}
We have shown that the constructed sequence $w$ satisfies the condition
$b)$, $w\in\mathcal{L}(P)$.

Finally let us verify that $w$ satisfies condition $c)$. Consider
an element $d\in D$ which is inserted into $w$ at some position
$k$, thus $d\equiv w_{k}$. We have ensured during the construction
process that $w_{k-1}<d\equiv w_{k}<w_{k+1}$, therefore condition
$1)$ of Def.~\ref{def:deletable} is satisfied. If $w_{j}<w_{k}$
for all $j\in(0,k)$, then condition $2a)$ of Def.~\ref{def:deletable}
is satisfied and $d$ is deletable in $w$. Otherwise, the set $L=\left\{ i\,\vert\,i<k\text{ and }w_{i}>w_{k}\right\} $
is nonempty. Let $l=\text{max }L$. It is clear that $w_{i}<w_{k}$
for all $i\in(l,k)$. Thus, if we can find a value of $j\in(l,k)$
such that $\omega^{-1}(w_{j})<_{P}\omega^{-1}(w_{k})$, then condition
$2b)$ of Def.~\ref{def:deletable} is satisfied. We show below that
indeed this is the case.

It follows directly from the definition of $l$ that $w_{l}>w_{k}$
and simultaneously $w_{k}>w_{l+1}$ (since $l+1\in\left(l,k\right)$);
therefore $w_{l}>w_{l+1}$, and $l$ is a descent. We already have
seen that condition $1)$ of Def.~\ref{def:deletable} is satisfied
for all the elements of $D$; therefore $w_{l}$ and $w_{l+1}$ are
not in $D$, but appear somewhere in the original sequence $v$, in
the form $w_{l}\equiv v_{\tilde{l}}$ and $w_{l+1}\equiv v_{\tilde{l}+1}$
for some $\tilde{l}\in[\,q\,]$. Since $l+1<k$, during Step~1 $d$
must have been assigned to some $D^{i_{d}}$ with $\tilde{l}<i_{d}$. 

Let us assume that $j_{d}<\tilde{l}$; we will see that this assumption
leads to a contradiction. From $\omega^{-1}(v_{j_{d}})<_{P}\omega^{-1}(d)$
it follows that $v_{j_{d}}<d$. By construction of $\tilde{l}$, we
have $d<v_{\tilde{l}}$. Therefore, if $j_{d}<\tilde{l}$, then there
must be a $i\in(j_{d}-1,\tilde{l})$ such that $v_{i}<d<v_{i+1}$.
Then, by definition of $i_{d}$, we would have $i_{d}\leq i<\tilde{l}$,
in contradiction with $\tilde{l}<i_{d}$. Therefore, the assumption
$j_{d}<\tilde{l}$ made at the beginning of this paragraph must be
wrong and we have $\tilde{l}\leq j_{d}$. Since $v_{\tilde{l}}>d$,
we find $\omega^{-1}(v_{\tilde{l}})\nless_{P}\omega^{-1}(d)$, and
therefore (by Eq.~(\ref{eq:jd})) $j_{d}\neq\tilde{l}$. This reasoning
shows that $\tilde{l}<j_{d}$.

The entry $v_{j_{d}}$ of $v$ appears in $w$ in the form $v_{j_{d}}\equiv w_{j}$
at some position $j\in[\,p\,]$. Since $\tilde{l}<j_{d}\leq i_{d}$
and elements of $v$ appear in the same order in $w$, we find $l<j<k$.
As we have shown earlier, by construction of $l$, we have $w_{m}<w_{k}$
for all $m\in(l,k)$ (and thus especially for all $m\in(j,k)$); and
by construction of $j$, we have $\omega^{-1}(w_{j})<_{P}\omega^{-1}(w_{k})$.
Therefore, condition $2b)$ of Def.~\ref{def:deletable} is satisfied.
It follows that every $d\in D$ is deletable in $w$, meaning that
$D\subset\text{Del}_{P}(w)$.

We have demonstrated that there is exactly one sequence $w$ of elements
of $\left[p\right]$ that satisfies conditions $a)-c)$ given at the
beginning of this proof, i.e., there is exactly one linear extension
$w$ such that $v=w\setminus D$ and $D\subset\text{Del}_{P}(w)$.
\end{proof}
\begin{example}
\noindent Let us demonstrate the insertion process described in Steps
1\&2 during the proof above. Consider the sequence $v=17536\in\mathcal{L}(P\setminus D)$
with the poset $P=\boldsymbol{3}\times\boldsymbol{3}$ shown in Fig.~\ref{fig:3x3 Hasse}
and the set of deleted elements $D=\left\{ 2,4,8,9\right\} $. For
the labels $d=2$ and $4$, we find for the set in Eq.~(\ref{eq:jd})
$\left\{ j\in[\,5\,]\,\vert\,\omega^{-1}(v_{j})<_{P}\omega^{-1}(d)\right\} =\left\{ j\in[\,5\,]\,\vert\,v_{j}\in\left\{ 1\right\} \right\} =\left\{ 1\right\} $,
and thus $j_{2}=j_{4}=1$. Since $v_{1}=1<2,4<v_{2}=7$, in Eq.~(\ref{eq:id})
we find $i_{2}=i_{4}=1$. Therefore, the labels $2$ and $4$ are
assigned to the subset $D^{1}$, and will be inserted into the sequence
$v$ between $v_{1}=1$ and $v_{2}=7$. For the label $8$, we find
that in Eq.~(\ref{eq:jd}), $\left\{ j\in[\,5\,]\,\vert\,\omega^{-1}(v_{j})<_{P}\omega^{-1}(8)\right\} =\left\{ j\in[\,5\,]\,\vert\,v_{j}\in\left\{ 5,7\right\} \right\} =\left\{ 2,3\right\} $,
and thus $j_{8}=3$. Since however the label $8$ is larger than any
of the following labels in $v$, $v_{3}=5$, $v_{4}=3$ and $v_{5}=6$,
we find in Eq.~(\ref{eq:id}) that $\left\{ i\,\vert\,j_{8}=3\leq i\leq5,\,v_{i}<8<v_{i+1}\right\} =\left\{ 5\right\} $
and therefore $i_{8}=5$. Finally, for the label $9$, we find that
in Eq.~(\ref{eq:jd}), $\left\{ j\in[\,5\,]\,\vert\,\omega^{-1}(v_{j})<_{P}\omega^{-1}(9)\right\} =\left\{ j\in[\,5\,]\,\vert\,v_{j}\in\left\{ 6,8\right\} \right\} =\left\{ 5\right\} $,
and thus $j_{9}=5$ and $i_{9}=5$. To summarize, in Step~1, we split
the set $D=\left\{ 2,4,8,9\right\} $ into the subsets $D^{0}=\varnothing$,
$D^{1}=\left\{ 2,4\right\} $, $D^{2}=\varnothing$, $D^{3}=\varnothing$,
$D^{4}=\varnothing$ and $D^{5}=\left\{ 8,9\right\} $. In Step~2,
the elements of each subset are arranged into a growing sequence,
specifically $d_{1}^{1}d_{2}^{1}=24$ and $d_{1}^{5}d_{2}^{5}=89$,
and inserted into $v$: 
\begin{figure}[H]
\noindent \centering{}\includegraphics[scale=0.5]{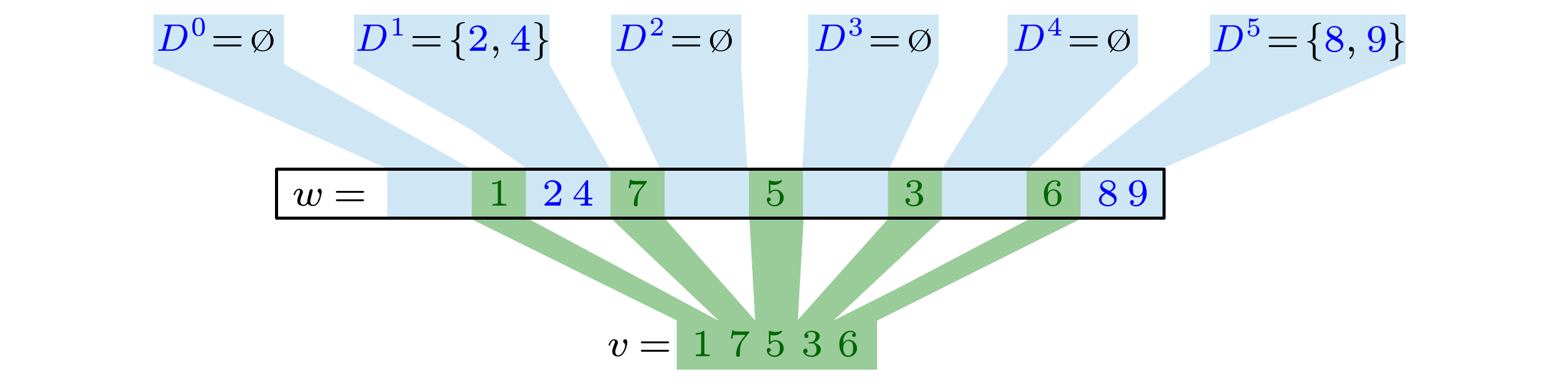}
\end{figure}
The resulting sequence $w=124753689$ was previously considered and
illustrated in Fig.~\ref{fig:six-linear-extensions}~$(e)$.
\end{example}

\begin{lem}
\label{Equiv<}Let $w$ be a linear extension of $P$ with the set
of deletable elements $\text{Del}_{P}(w)$. For every set $D\subset\text{Del}_{P}(w)$,
the sequence $w\setminus D$ is a linear extension of $P\setminus D$.
\end{lem}

\begin{proof}
Consider two elements $s,t\in P\setminus D$ with $s<_{P\setminus D}t$.
In order to show that $w\setminus D$ is a linear extension of $P\setminus D$,
we have to demonstrate that $\omega(s)$ appears in $v$ before $\omega(t)$.
Since $P\setminus D$ is an induced subposet of $P$, it follows from
$s<_{P\setminus D}t$ that $s<_{P}t$, and since $w$ is a linear
extension, this implies that $\omega(s)$ precedes $\omega(t)$ in
$w$. Clearly then, by construction of $v$, $\omega(s)$ also precedes
$\omega(t)$ in $v$.
\end{proof}
By combining the previous two lemmata, we find that
\begin{lem}
\label{all-lin-ext}The union of the Jordan-Hölder sets of all subposets
of $P$ is given by
\[
\bigcup_{Q\subset P}\mathcal{L}(Q)=\bigcup_{w\in\mathcal{L}(P)}\bigcup_{D\subset\text{Del}_{P}(w)}\left\{ w\setminus D\right\} .
\]
\end{lem}

\begin{proof}
Consider first a set $D\subset[\,p\,]$ and the corresponding subposet
of $P$ given by $P\setminus D$. It follows directly from Lemmata~\ref{Equiv>}
and \ref{Equiv<} that the collection of linear extensions of $P\setminus D$
is can be written as
\[
\mathcal{L}(P\setminus D)=\bigcup_{\substack{w\in\mathcal{L}(P)\\
\text{for which}\\
D\subset\text{Del}_{P}(w)
}
}\left\{ w\setminus D\right\} .
\]
Therefore, the set of linear extensions of subposets of $P$ is given
by
\[
\bigcup_{Q\subset P}\mathcal{L}(Q)=\bigcup_{D\subset[\,p\,]}\mathcal{L}(P\setminus D)=\bigcup_{D\subset[\,p\,]}\bigcup_{\substack{w\in\mathcal{L}(P)\\
\text{for which}\\
D\subset\text{Del}_{P}(w)
}
}\left\{ w\setminus D\right\} =\bigcup_{w\in\mathcal{L}(P)}\bigcup_{D\subset\text{Del}_{P}(w)}\left\{ w\setminus D\right\} .
\]
\end{proof}

\section{Proof of Theorem~\emph{\ref{Extended-order-polynomial}}}

We are now ready to combine the so far derived lemmata into the derivation
of the closed from of the extended order polynomial given in the first
section.
\begin{proof}
\emph{(of Theorem~\ref{Extended-order-polynomial})} By application
of the Definitions given in Eqs.~(\ref{eq:StrictOrderPoly}) and
(\ref{eq:GenOrderPoly}) as well as (in line 3) Lemmata~\ref{same-nr-descents}
and \ref{all-lin-ext}, we find
\begin{eqnarray*}
\text{E}_{P}^{\circ}(n,z) & = & \sum_{Q\subset P}\Omega_{Q}^{\circ}(n)z^{\#Q}\\
 & = & \sum_{Q\subset P}\sum_{v\in\mathcal{L}(Q)}\binom{n+\text{des}(v)}{\#Q}z^{\#Q}\\
 & = & \sum_{w\in\mathcal{L}(P)}\sum_{D\subset\text{Del}_{P}(w)}\binom{n+\text{des}(w)}{p-\#D}z^{p-\#D}\\
 & = & \sum_{w\in\mathcal{L}(P)}\sum_{k=0}^{p}\#\left\{ D\subset\text{Del}_{P}(w)\,\vert\,\#D=k\right\} \cdot\binom{n+\text{des}(w)}{p-k}z^{p-k}\\
 & = & \sum_{w\in\mathcal{L}(P)}\sum_{k=0}^{p}\binom{\text{del}_{P}(w)}{k}\cdot\binom{n+\text{des}(w)}{p-k}z^{p-k}
\end{eqnarray*}
By inverting the order of summation in the inner sum, we obtain Eq.~(\ref{eq:Znz}).
\end{proof}

\section{Perspectives}

Our main motivation to develop the extended strict order polynomial
$\text{E}_{P}^{\circ}(n,z)$ introduced in the current communication
is its close relation to the Zhang-Zhang polynomial \cite{zhang1996theclar,zhang1996theclar2,zhang2000theclar}
enumerating Clar covers of benzenoid hydrocarbons \cite{clar1972thearomatic},
a topic to which we devoted feverish activity in our laboratory for
almost a decade now \cite{chou2012analgorithm,chou2014zzdecomposer,chou2014determination,witek2015zhangzhang,witek2017zhangzhang,langner2019IFTTheorems5,langner2019BasicApplications6}.
Our recent contribution, introducing the interface theory of benzenoids
\cite{langner2019IFTTheorems5,langner2019BasicApplications6}, demonstrated
that enumeration of Clar covers of a benzenoid $\boldsymbol{B}$ can
be efficiently achieved by studying distributions of covered interface
edges in interfaces of $\boldsymbol{B}$. The relative positions of
the covered edges can be expressed in a form of a poset. Without giving
too many unnecessary details, we can say that for a regular benzenoid
strip $\boldsymbol{B}$ of length $n$, there exists a poset $P$
such that the extended strict order polynomial $\text{E}_{P}^{\circ}(n,z)$
coincides with the Zhang-Zhang polynomial $\textrm{ZZ}\left(\boldsymbol{B},x\right)$
of $\boldsymbol{B}$ (with $z=x+1$). A detailed proof of this fact
is quite technical and will be announced soon. The announced equivalence
between the extended strict order polynomials $\text{E}_{P}^{\circ}(n,z)$
developed in the current study and the Zhang-Zhang polynomials $\textrm{ZZ}\left(\boldsymbol{B},x\right)$
of regular benzenoid strips $\boldsymbol{B}$ allows us to recognize
(currently without a formal proof) a large collection of facts about
$\text{E}_{P}^{\circ}(n,z)$ due to the previously discovered facts
about the ZZ polynomials. Among others, the following facts are easy
to deduce:
\begin{enumerate}
\item The chain $P=\boldsymbol{p}$ corresponds to a parallelogram $M\left(p,n\right)$
\begin{figure}[H]
\noindent \centering{}\includegraphics[scale=0.45]{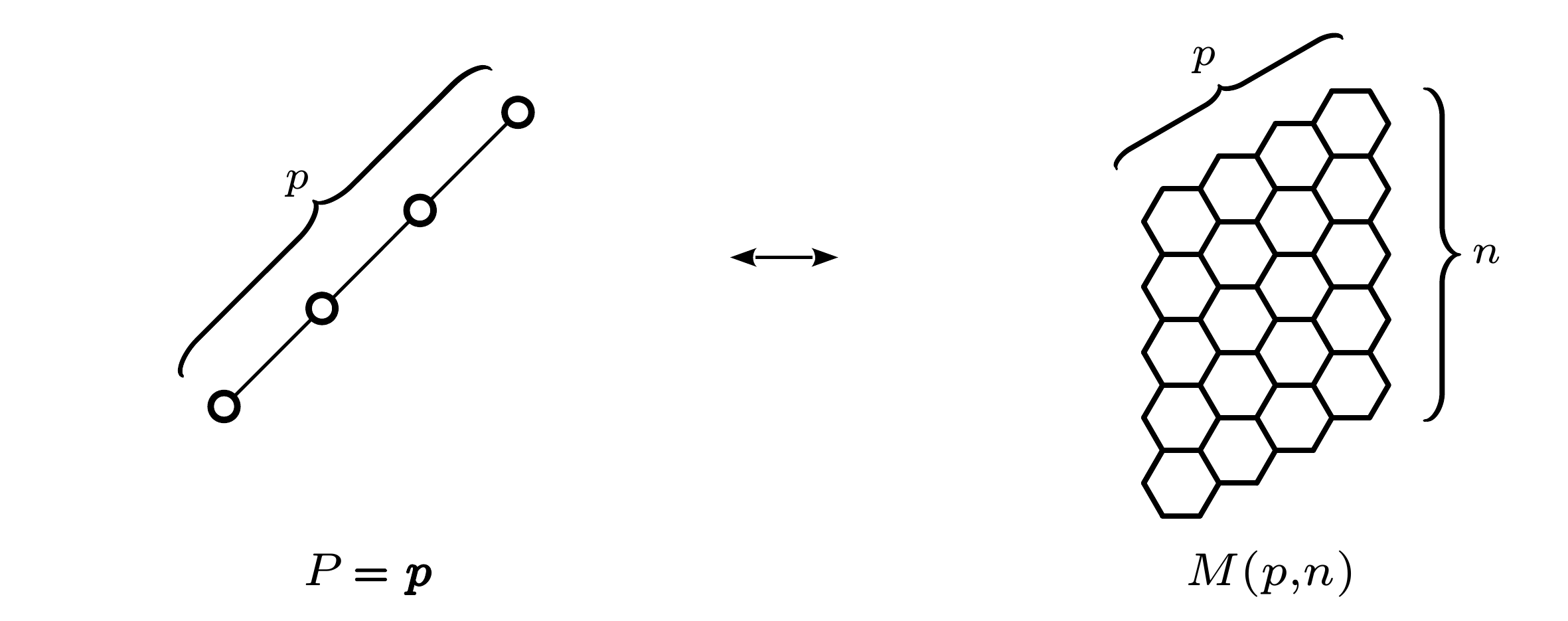}
\end{figure}
for which the ZZ polynomial is given in form of a hypergeometric function,
$\textrm{ZZ}\left(M\left(m,n\right),x\right)={}_{2}F_{1}\negthinspace\left[\begin{array}{c}
-m,-n\\
1
\end{array};x+1\right]$ \cite{gutman2006zhangtextendash,chou2012analgorithm,chou2014closedtextendashform};
consequently, we have
\begin{equation}
\text{E}_{\boldsymbol{p}}^{\circ}(n,z)={}_{2}F_{1}\negthinspace\left[\begin{array}{c}
-p,-n\\
1
\end{array};z\right].\label{eq:chainEp}
\end{equation}
This result is also directly obvious from Theorem~\ref{Extended-order-polynomial}:
The Jordan-Hölder set of $\boldsymbol{p}$ consists of only one element,
$\mathcal{L}(\boldsymbol{p})=\{123\ldots p\}$, for which $\text{del}_{\boldsymbol{p}}(123\ldots p)=p$
and $\text{des}_{\boldsymbol{p}}(123\ldots p)=0$. Thus, Eq.~(\ref{eq:Znz})
immediately assumes the form of Eq.~(\ref{eq:chainEp}).
\item The poset $P$ containing $p$ non-comparable elements corresponds,
according to the interface theory of benzenoids, to a prolate rectangle
$Pr\left(p,n\right)$
\begin{figure}[H]
\noindent \centering{}\includegraphics[scale=0.45]{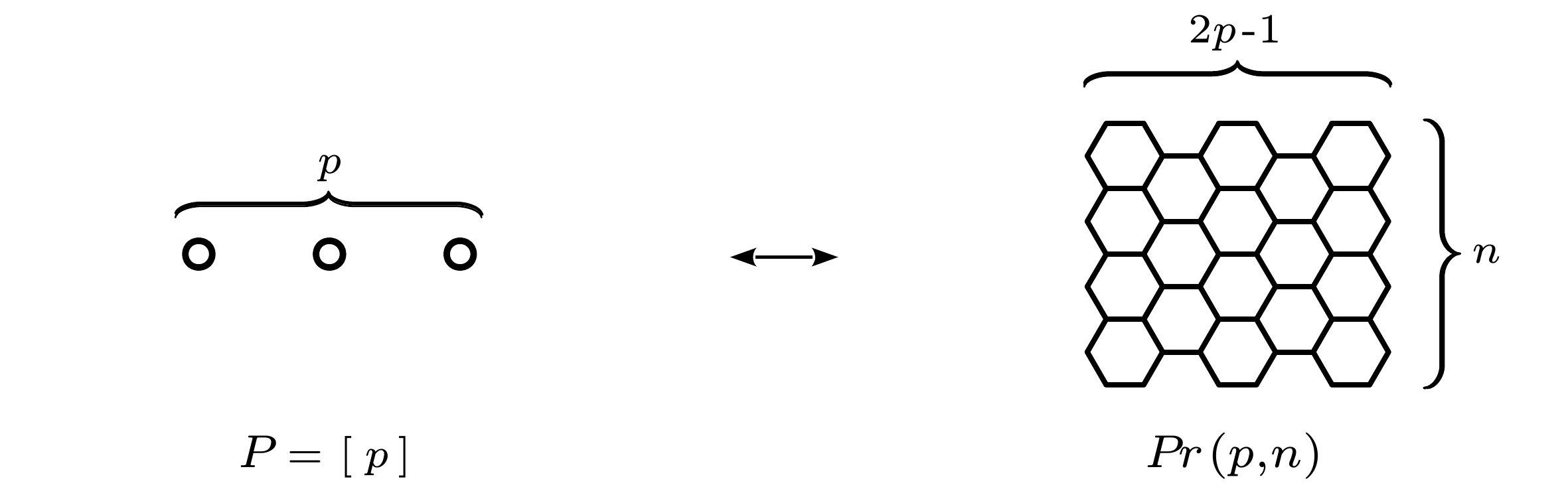}
\end{figure}
for which the ZZ polynomial is given by $\textrm{ZZ}\left(Pr\left(m,n\right),x\right)=\left(1+n\left(x+1\right)\right)^{m}$
\cite{zhang2000theclar,chou2012zhangzhang,chou2016closedform}; consequently,
we have
\begin{equation}
\text{E}_{[\,p\,]}^{\circ}(n,z)=\left(1+nz\right)^{p}.
\end{equation}
\item The poset $P=\boldsymbol{2}\times\boldsymbol{m}$ corresponds to a
hexagonal graphene flake $O\left(2,m,n\right)$
\begin{figure}[H]
\noindent \centering{}\includegraphics[scale=0.45]{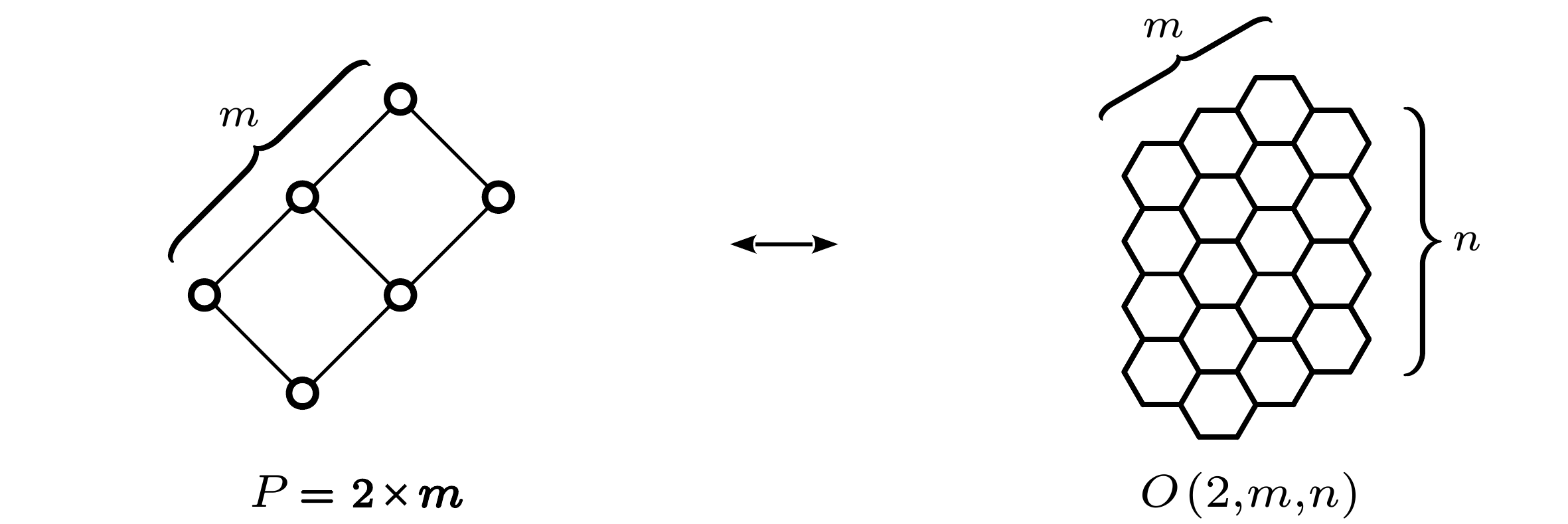}
\end{figure}
It follows from the ZZ polynomial $\text{ZZ}\left(O\left(2,m,n\right),x\right)$
\cite{He2020determinantal,He2020JohnSachs,witek2020overlappingparas}
that the strict order polynomial has the form of a $2\times2$ determinant
\begin{equation}
\text{E}_{\boldsymbol{2}\times\boldsymbol{m}}^{\circ}(n,z)=\left|\begin{array}{ll}
\phantom{z\binom{n+1}{2}\,}{}_{2}F_{1}\negthinspace\left[\begin{array}{c}
\phantom{1}-m,\phantom{1}-n\\
1
\end{array};z\right] & \,\,\,z\binom{m+1}{2}\,{}_{2}F_{1}\negthinspace\left[\begin{array}{c}
1-m,1-n\\
3
\end{array};z\right]\\
\\
z\binom{n+1}{2}\,{}_{2}F_{1}\negthinspace\left[\begin{array}{c}
1-m,1-n\\
3
\end{array};z\right] & \,\,\,\phantom{z\binom{m+1}{2}\,}{}_{2}F_{1}\negthinspace\left[\begin{array}{c}
\phantom{1}-m,\phantom{1}-n\\
1
\end{array};z\right]
\end{array}\right|
\end{equation}
\item The strict order polynomial for the lattice $P=\boldsymbol{l}\times\boldsymbol{m}$
is unknown, following the fact that this poset corresponds to the
hexagonal flake $O\left(l,m,n\right)$
\begin{figure}[H]
\noindent \centering{}\includegraphics[scale=0.45]{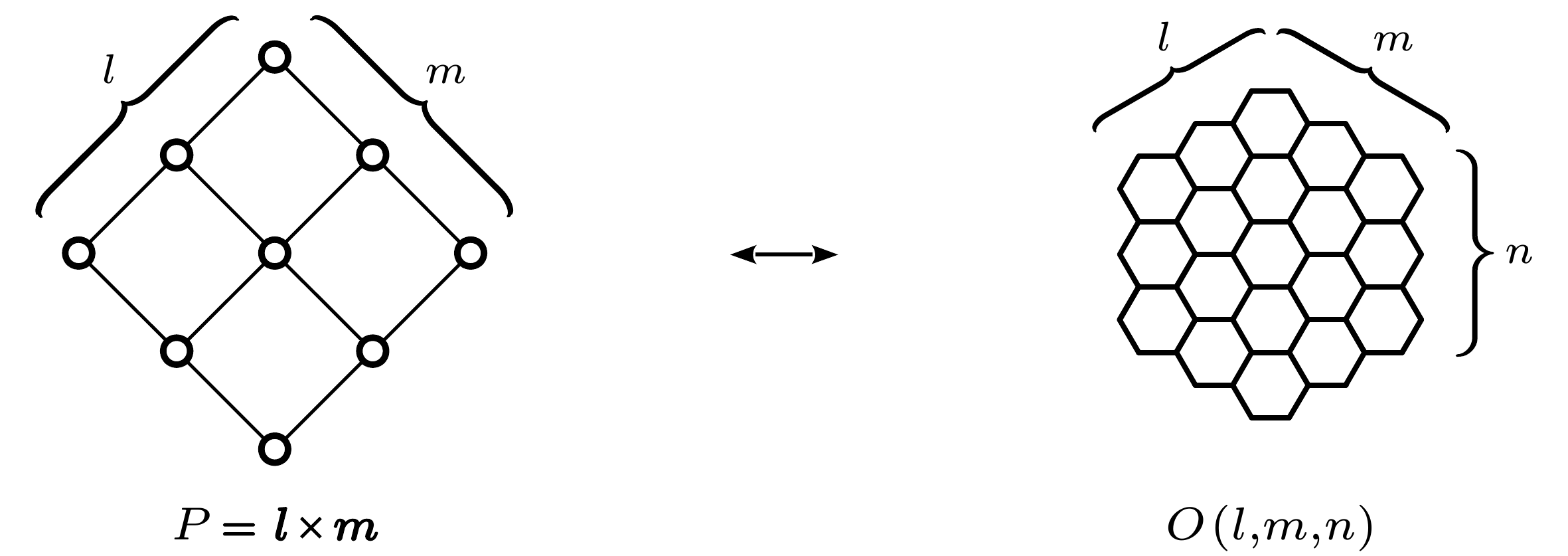}
\end{figure}
The ZZ polynomial $\text{ZZ}\left(O\left(l,m,n\right),x\right)$ of
this structure constitutes the hardest unsolved problem in the theory
of ZZ polynomials \cite{chou2012zhangzhang,chou2014determination,He2020determinantal,witek2020flakes1-8}.
\item The fence $P=Q(1,m)$ with $m$ elements corresponds to a zigzag chain
$Z(m,n)$
\begin{figure}[H]
\noindent \centering{}\includegraphics[scale=0.45]{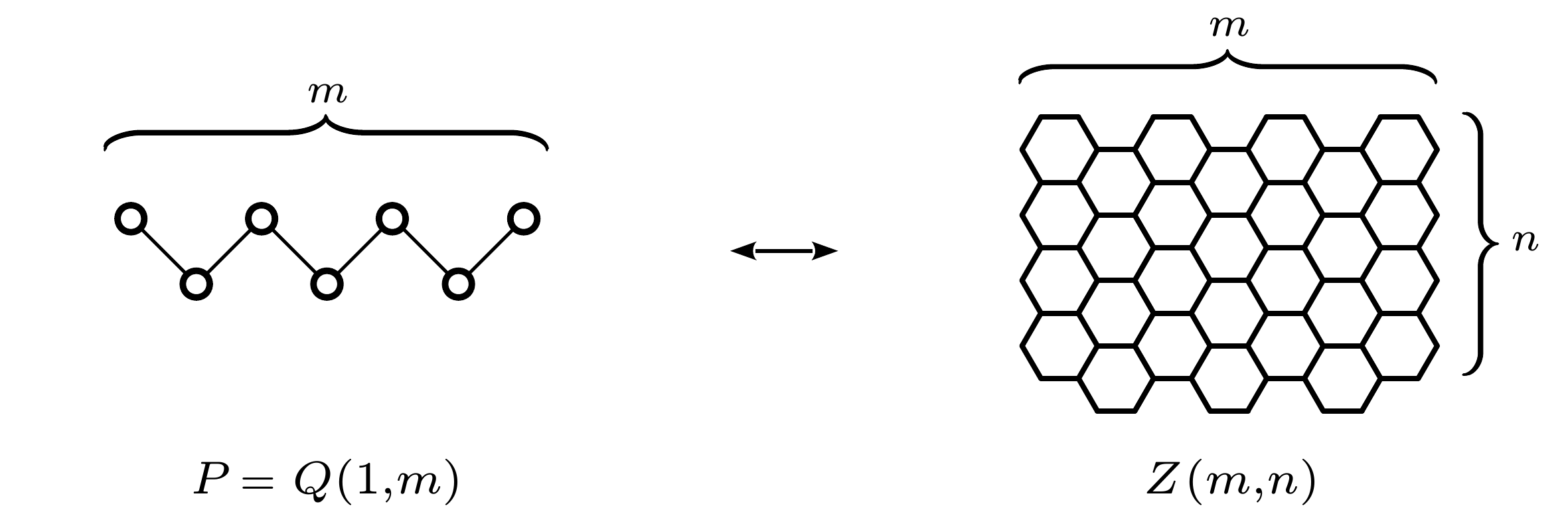}
\end{figure}
 The expression for $\text{ZZ}\left(Z\left(m,n\right),x\right)$\textemdash and
consequently, $\text{E}_{P}^{\circ}(n,z)$\textemdash is given by
a very lengthy formula \cite{chou2012zhangzhang,chou2014determination,langner2018multiplezigzag1},
but the associated generating function has the form of a continued
fraction \cite{langner2018multiplezigzag1}
\begin{figure}[H]
\noindent \centering{}\begin{equation} \includegraphics[scale=0.50]{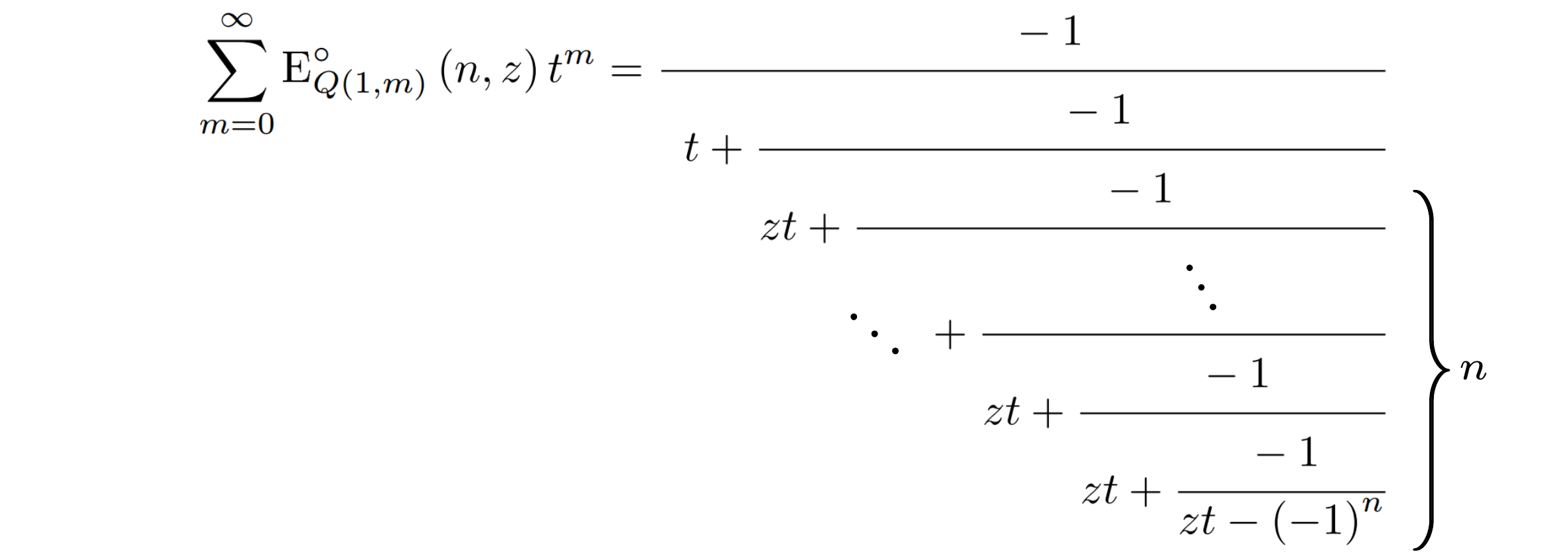}\label{eq:genZZZb} \end{equation}
\end{figure}
An analogous generating function with respect to $n$ is unknown.
\end{enumerate}
The extended order polynomial $\text{E}_{P}^{\circ}(n,z)$ can be
also computed in an efficient fashion directly from Eq.~(\ref{eq:Znz})
through an algorithm based on a graph of ,,compatible'' antichains
of $P$. Propagating weights through this graph in a certain way yields
the extended order polynomial without ever having to construct the
entire set $\mathcal{L}(P)$. This algorithm has been implemented
in Maple 16 \cite{maple16} and will be reported later. For the example
of the poset $P=\boldsymbol{3}\times\boldsymbol{3}$ depicted in Fig.~\ref{fig:3x3 Hasse},
we obtain in this way
\begin{eqnarray}
\text{E}_{\boldsymbol{3}\times\boldsymbol{3}}^{\circ}(n,z) & = & \sum_{k=0}^{9}\left(\binom{9}{k}\binom{n}{k}+\left(9\binom{9-2}{k-2}+\binom{9-3}{k-3}\right)\binom{n+1}{k}\right.\label{eq:EP3x3}\\
 &  & \hphantom{\sum_{k=0}^{9}\,}+\left(\binom{9-3}{k-3}+17\binom{9-4}{k-4}+2\binom{9-5}{k-5}\right)\binom{n+2}{k}\nonumber \\
 &  & \hphantom{\sum_{k=0}^{9}}\left.\,+\left(2\binom{9-5}{k-5}+7\binom{9-6}{k-6}+\binom{9-7}{k-7}\right)\binom{n+3}{k}+\binom{9-7}{k-7}\binom{n+4}{k}\right)z^{k}.\nonumber 
\end{eqnarray}
We suspect that the coefficients $e_{l,j}\left(P\right)$ appearing
in $\text{E}_{P}^{\circ}(n,z)$ in front of the terms $\binom{9-2l-j}{k-2l-j}\binom{n+l}{k}$
are \#P-complete to compute, in close analogy to the coefficients
$e\left(P\right)$ corresponding to the number of linear extensions
of $P$. These coefficients are growing very fast with the size of
the poset $P$. The largest of the coefficients $e_{l,j}\left(\boldsymbol{3}\times\boldsymbol{3}\right)$
is only 17 (as can be easily seen from Eq.~(\ref{eq:EP3x3})), but
larger $P$ are characterized by much greater coefficients, e.g.,
$\max\text{\ensuremath{\left(e_{l,j}\left(\boldsymbol{4}\times\boldsymbol{4}\right)\right)}=3765}$,
 $\max\text{\ensuremath{\left(e_{l,j}\left(\boldsymbol{4}\times\boldsymbol{5}\right)\right)}=200440}$,
$\max\text{\ensuremath{\left(e_{l,j}\left(\boldsymbol{5}\times\boldsymbol{5}\right)\right)}=61885401}$,
and $\max\text{\ensuremath{\left(e_{l,j}\left(\boldsymbol{5}\times\boldsymbol{6}\right)\right)}=}27950114975$.

It seems that the introduced here extended strict order polynomial
$\text{E}_{P}^{\circ}(n,z)$ can be immediately generalized to the
non-strict case using the reciprocity theorem of Stanley (Corollary~3.15.12
of \cite{stanley_enumerative_1986}), but this problem is not pursued
here further.

\end{document}